\documentclass[a4paper, bibliography=totoc]{scrartcl}
\usepackage{amsmath}
\usepackage{amssymb}
\usepackage{enumerate}
\usepackage{amsthm}
\usepackage{mathtools} % for DeclarePairedDelimiter
\usepackage[UKenglish]{babel}
\usepackage[T1]{fontenc}
\usepackage[utf8]{inputenc}
\usepackage{tikz}
\usepackage{thmtools} %needed for thm-restate package
\usepackage{thm-restate} %to allow restating a theorem
\usepackage{hyperref} % Usually should be declared as the last package
\usepackage{color}
\definecolor{grey}{rgb}{.7,.7,.7}
\definecolor{blue}{rgb}{0,0,.8}
\definecolor{red}{rgb}{.8,0,0}
\definecolor{green}{rgb}{0,.4,0}
\definecolor{gold}{rgb}{0.8,0.6,0.1}
\definecolor{brown}{rgb}{0.8,0.4,0.1}
\def\marrow{{\marginpar[\hfill$\Rrightarrow$]{$\Lleftarrow$}}}

\def\md#1{\color{blue} {\textsc MD: }{\marrow\textsf #1} \normalcolor}

\usepackage[open]{bookmark} %overrides warnings caused by the fact that I skip some elements of the hierarchy \section, \subsection, \subsubsection, \paragraph

\declaretheorem[name=Theorem,numberwithin=section]{theorem} %Weird thing, the package thmtools ``requires'' it.

\newtheorem*{definition*}{Definition}
\newtheorem{definition}[theorem]{Definition}

\newtheorem*{lemma*}{Lemma}
\newtheorem{lemma}[definition]{Lemma}

\newtheorem*{algorithm*}{Algorithm}

\newtheorem*{construction*}{Construction}

\newtheorem*{proposition*}{Proposition}
\newtheorem{proposition}[definition]{Proposition}

\newtheorem*{obs*}{Observation}

\newtheorem*{fact*}{Fact}

\newtheorem*{remark*}{Remark}
\newtheorem{remark}[definition]{Remark}

\newtheorem*{quest*}{Question}

\newtheorem*{corollary*}{Corollary}
\newtheorem{corollary}[definition]{Corollary}

\newtheorem*{conjecture*}{Conjecture}

\newtheorem*{question*}{Question}

\newtheorem*{example*}{Example}

\newtheorem*{hypothesis*}{Hypothesis}
\newtheorem{hypothesis}[definition]{Hypothesis}

\newcommand{\BM}{\operatorname{BM}}
\newcommand{\C}{\mathcal{C}}

\newcommand{\R}{\mathbb{R}}

\newcommand{\lip}{\operatorname{Lip}}

\newcommand{\dist}{\operatorname{dist}}
\newcommand{\N}{\mathbb{N}}

\newcommand{\leb}{\mathcal{L}}

\DeclareMathOperator{\diam}{diam}

\newcommand{\abs}[1]{\left|#1\right|}
\newcommand{\lnorm}[2]{\left\|#2\right\|_#1}
\newcommand{\norm}[1]{\left\|#1\right\|}
\newcommand{\skp}[1]{\langle#1\rangle}

\newcommand{\proj}{\operatorname{proj}}
\newcommand{\Diff}{\operatorname{Diff}}

\newcommand{\set}[1]{\left\{#1\right\}}

\newcommand{\Exp}{\mathbb{E}}
\newcommand{\hl}{}
\newcommand{\Graph}{\operatorname{Graph}}

\makeatletter
\newcommand{\mylabel}[2]{#2\def\@currentlabel{#2}\label{#1}}
\makeatother

\title{Typical differentiability within an exceptionally small set.}
\author{Michael Dymond\thanks{The author acknowledges the support of Austrian Science Fund (FWF): P 30902-N35.}}

\begin{document}
	\maketitle
\begin{abstract}
	We verify the existence of a purely unrectifiable set in which the typical Lipschitz function has a large set of differentiability points. The example arises from a construction, due to Csörnyei, Preiss and Ti\v ser, of a universal differentiability set in which a certain Lipschitz function has only a purely unrectifiable set of differentiability points.
\end{abstract}

\section{Introduction.}
Whilst Rademacher's Theorem asserts that any set of points of non-differentiability of a Lipschitz function on Euclidean space is null, the sets most neglible from the point of view of differentiability problems are, as described in the work \cite{acp2010differentiability} of Alberti, Csörnyei and Preiss, those sets in which some Lipschitz function fails to have a single directional derivative. In this paper we show that even these most exceptional sets can nonetheless provide surprisingly many points of differentiability for surprisingly many Lipschitz functions. 

In \cite{acp2010differentiability} it is established that the negligible sets referred to above are precisely the class of \emph{uniformly purely unrectifiable} sets. A subset $P$ of Euclidean space is said to be \emph{purely unrectifiable} if $P$ intersects every $\mathcal{C}^{1}$ curve in a set of one-dimensional Lebesgue measure zero. The class of \emph{uniformly} purely unrectifiable sets are defined according to a formally stronger condition (see \cite[Definition~1.4 and Remark~1.7]{maleva_preiss2018}) and for a significant time it remained an open question whether these two classes coincide. However, a recent announcement of M\'athe answers this question positively for Borel sets (\cite[Remark~1.7]{maleva_preiss2018}). In the present work, we adopt the convention of restricting both notions to Borel sets, that is, we add Borel as a condition to the definitions of pure and uniform pure unrectifiability. Thus, the notions of pure und uniform pure unrectifiability coincide and we will, from this point onwards, refer only to purely unrectifiable sets.

Current investigations of purely unrectifiable sets have established that these sets are most exceptional with respect to differentiability, not only in the sense of non-availability of directional derivatives. Preiss and Maleva prove in \cite[Theorem~1.13]{maleva_preiss2018} that any purely unrectifiable set is contained in a set \hl{of} points where non-differentiability of some Lipschitz function occurs in its strongest possible form. Any purely unrectifiable set $P\subseteq \R^{d}$ admits a $1$-Lipschitz function $f\colon \R^{d}\to\R$ such that for every $x\in P$ \emph{every} linear mapping $\R^{d}\to\R$ with norm at most one masquerades as the derivative of $f$ at $x$. More precisely,
\begin{equation*}
\liminf_{r\to 0}\sup_{\norm{y}\leq r}\frac{\abs{f(x+y)-f(x)-\langle{e,y}\rangle}}{r}=0
\end{equation*}
holds for all $x\in P$ and $e\in\R^{n}$ with $\norm{e}\leq 1$. 
%Further, Merlo~\cite{merlo} proves that for a given $F_{\sigma}$ purely unrectifiable set $P$, the typical Lipschitz mapping in an appropriate space has no directional derivatives inside $P$.

%Merlo's result and the notion of a typical Lipschitz mapping require further explanation. 
\hl{An even more hostile class of sets for differentiability could be loosely defined as those sets in which not just `some', but rather `many' Lipschitz functions fail to have points of differentiability. In the 1990's, Preiss and Ti\v ser~\cite{preiss_tiser94} characterised those analytic subsets of the interval $[0,1]$ in which the typical Lipschitz function $[0,1]\to\R$ has no points of differentiability. Very recently, the author and Maleva~\cite{dymond_maleva2019dichotomy} generalised this result to spaces of Lipschitz functions $[0,1]^{d}\to\R$ for all Euclidean dimensions $d\geq 1$. We discuss these works in more detail shortly, but first, let us make precise, what is meant by a typical Lipschitz function:} In what follows we consider for a compact metric space $K$ the space $\lip_{1}(K,\R^{l})$ of Lipschitz mappings $f\colon K\to\R^{l}$ with $\lip(f)\leq 1$. When $l=1$, as it will be for almost all of this work, we shorten the notation to $\lip_{1}(K)$. We view $\lip_{1}(K,\R^{l})$ as a complete metric space equipped with the supremum metric
\begin{equation*}
d_{\infty}(f,g):=\lnorm{\infty}{f-g},\qquad f,g\in\hl{\lip_{1}(K,\R^{l})}.
\end{equation*}
The word typical is used in this paper in the sense of the Baire Category Theorem. Thus, we say that typical functions \hl{(or the typical function)} in $\lip_{1}(K)$ have \hl{(has)} a certain property if the set of those functions having that property is a residual subset of $\lip_{1}(K)$.

\hl{For a compact metric space $K$, another natural means of giving the class of Lipschitz functions $K\to\R$ a complete metric space structure is to consider the space $\lip(K)$ of all such Lipschitz functions (not just those with Lipschitz constant at most one) equipped with the metric
	\begin{equation*}
	d_{\lip}(f,g)=\lnorm{\infty}{f-g}+\lip(f-g),\qquad f,g\in\lip(K,\R^{l}).
	\end{equation*}
However, this space has significantly less desirable properties. For a start, it is non-separable. Moreover, for differentiability questions (when say $K=[0,1]^{d}$), this space is much less appealing because smooth functions in this space are not dense, in fact differentiable functions form a nowhere dense, closed set, as discussed in \cite{preiss_tiser94}.}

\hl{For} Lipschitz functions on the \hl{Euclidean cube}, i.e. in the function \hl{spaces $\lip_{1}([0,1]^{d})$ for $d\in\N$}, \hl{differentiability of the typical function inside analytic sets} is well understood, due to the \hl{aforementioned} works \cite{preiss_tiser94} and \cite{dymond_maleva2019dichotomy}. In the former, Preiss and Ti\v ser characterise \hl{analytic} subsets of the interval $[0,1]$ in which the typical function $f\in\lip_{1}([0,1])$ is nowhere differentiable; they prove that the sets with this property are precisely those contained in an $F_{\sigma}$ set of Lebesgue measure zero. \hl{In the latter, the author and Maleva generalise this characterisation to all Euclidean dimensions: they prove that an analytic subset of $[0,1]^{d}$ contains no points of differentiability of the typical function in $\lip_{1}([0,1]^{d})$ if and only if it can be covered by countably many closed, purely unrectifiable sets. This statement forms one half of a dichotomy of analytic sets established in \cite{dymond_maleva2019dichotomy}. To complete the dichotomy, the author and Maleva~\cite{dymond_maleva2019dichotomy} show that any analytic set failing the above coverability condition captures points of differentiability of the typical function in $\lip_{1}([0,1]^{d})$. Merlo~\cite{merlo}, another very recent work, proves a dichotomy of a similar nature, with differentiability replaced by directional differentiability. Additionally, Merlo~\cite{merlo} provides an independent proof of the non-differentiability part~\cite[Theorem~2.7]{dymond_maleva2019dichotomy}, of the dichotomy in \cite{dymond_maleva2019dichotomy}.} 

In particular, the result \cite[Theorem~2.1]{dymond_maleva2019dichotomy} permits examples of purely unrectifiable sets inside $(0,1)^{d}$ in which the typical $f\in \lip_{1}([0,1]^{d})$ has a point of differentiability. Indeed, any relatively residual and null subset of some line segment in $(0,1)^{d}$ would provide such an example. This is a \hl{somewhat} surprising outcome: Purely unrectifiable sets are so tiny that they see only the most terrible occurences of non-differentiability of some Lipschitz function. However, these exceptional sets may nonetheless capture points of differentiability of very many Lipschitz functions. 

Although the results of \hl{\cite{dymond_maleva2019dichotomy}} may be used to verify existence of purely unrectifiable sets capturing a point of differentibility of the typical Lipschitz function, they do not allow for any non-trivial, \hl{measure-theoretic, lower} bound~\footnote{\hl{For an incomparable topological description of the size of captured sets of differentiability points, see \cite[Remark~2.9]{dymond_maleva2019dichotomy}.}}. on the size of the set of captured differentiability points. Due to the fundamental Besicovitch-Federer Projection Theorem~\hl{\cite[Theorem~18.1]{mattila_1995}}, one-dimensional Hausdorff measure is an important means of distinction between purely unrectifiable sets. The theorem implies that any purely unrectifiable set of $\sigma$-finite one-dimensional Hausdorff measure has projections of Lebesgue measure zero on almost every one-dimensional subspace. Our main result verifies the existence of a purely unrectifiable set in which the typical Lipschitz function has a particularly large set of differentiability points, where large is understood in the sense of the Besicovitch-Federer Projection Theorem.
\begin{theorem}\label{thm:pu_tds}
	There exists a \hl{(Borel)} purely unrectifiable set $P\subseteq [0,1]^{2}$ such that the typical function $f\in\lip_{1}([0,1]^{2})$ has points of differentiability in $P$ and moreover the set $\Diff(f)\cap P$ of these points is large in the following senses:
	\begin{enumerate}[(a)]
		\item\label{non_sig_fin} $\Diff(f)\cap P$ has non-$\sigma$-finite one dimensional Hausdorff measure.
		\item\label{all_proj} $\Diff(f)\cap P$ projects in every direction to a set of positive Lebesgue measure, that is,
		\begin{equation*}
		\leb\left(\langle{\Diff(f)\cap P,e}\rangle\right)>0
		\end{equation*}
		for every $e\in S^{1}$.
	\end{enumerate}	 
\end{theorem}
Note that \eqref{non_sig_fin} actually follows from \eqref{all_proj} via the Besicovitch-Federer Projection Theorem \hl{and the fact that sets $\Diff(f)$ of differentiability points are Borel (\cite[Corollary~3.5.5]{LPT2012frechet})}.

The proof of Theorem~\ref{thm:pu_tds} is based on the modern theory of universal differentiability sets which originates from the natural question of whether the classical Rademacher's Theorem for Lipschitz mappings admits a converse and the first negative answer to this question given by Preiss~\cite{PREISS1990312}. The natural converse to Rademacher's Theorem proposes that any Lebesgue null set $E\subseteq \R^{d}$ is contained in the set of non-differentiability points of some Lipschitz mapping $f\colon \R^{d}\to\R^{l}$. Whilst \cite{PREISS1990312} provides a counterexample for the case of real valued functions on the plane, i.e. the case $d=2$, $l=1$, major breakthroughs \cite{preiss_speight2013}, \cite{acp2010differentiability}, \cite{CJ}, in the last decade have now completely resolved the question for general dimensions. The converse is valid if and only if $l\geq d$, that is, if the dimension of the target space is at least that of the domain.

Thus, if $1\leq l<d$, the Euclidean space $\R^{d}$ contains Lebesgue null sets which capture a point of differentiability of \emph{every} Lipschitz mapping $\R^{d}\to\R^{l}$. Sets with the latter property are given the name \emph{universal differentiability sets}, first proposed in \cite{Dore_Maleva2}. These surprising objects have attracted much new research attention and have been studied in an array of different settings, for example Euclidean spaces (\cite{Dore_Maleva1}, \cite{Dore_Maleva2}, \cite{dymond_maleva2016}, \cite{dymond2017structure}, \cite{preiss_speight2013}), Banach spaces \cite{Dore_Maleva3}, and metric groups (\cite{pinamonti_speight2017uds}, \cite{ledonne_pinamonti_speight2017universal}).

For a given universal differentiability set $E\subseteq \R^{d}$ it is natural to ask how large the sets
\begin{equation*}
E\cap \Diff(f),\qquad f\colon\R^{d}\to\R,\,\text{Lipschitz},
\end{equation*}
are as subsets of $E$. Previous work~\cite{dymond2017structure} of the author verifies that these sets are large in a topological sense. Any universal differentiability set can be reduced to a `kernel' in which the set of differentiability points of any given Lipschitz function form a dense subset. In contrast, an example provided by Csörnyei, Preiss and Ti\v ser~\cite{CPT_2005}, demonstrates that these sets $E\cap \Diff(f)$ of captured differentiability points can be surprisingly tiny subsets of $E$ in a measure theoretic sense, namely they can be purely unrectifiable. Recall from previous discussion in this introduction that purely unrectifiable sets are very far away from being universal differentiability sets, hence purely unrectifiable subsets of $E$ can be thought of as small subsets. 

The aforementioned example of Cs\"ornyei, Preiss and Ti\v ser~\cite{CPT_2005} and its construction provide the basis of the proof of Theorem~\ref{thm:pu_tds}. The construction produces a universal differentiability set $E\subseteq\R^{2}$, a purely unrectifiable subset $P\subseteq E$ and a Lipschitz function $h\colon\R^{2}\to \R$ so that all differentiability points of $h$ in the set $E$ are contained in $P$. By modification of the construction, we ensure that the purely unrectifiable set $P$ additionally captures many points of differentiability of the typical Lipschitz function. Our argument stems from the idea that most points of non-differentiability of $h$ are preserved for the function $g+h$ for typical $g\in \lip_{1}([0,1]^{2})$. However, there are rather too many such points in the $G_\delta$, dense set $E$ given by \cite{CPT_2005} in order to preserve non-differentiability at all of them. Thus, we crucially pass to a compact universal differentiability set $\hl{Y}\subseteq E$, given by a construction of Dor\'e and Maleva in \cite{Dore_Maleva2}. In this much smaller set we are able to preserve non-differentiability of $h$ everywhere in the set $\hl{Y}\setminus P$ for functions $g+h$ for the typical $g\in\lip_{1}([0,1]^{2})$. Since $\hl{Y}$ is a universal differentiability set, this leads to the conclusion that $g+h$ has points of differentiability in $P$. In other words, $P$ captures points of differentiability of the typical Lipschitz function in the shifted space $h+\lip_{1}([0,1]^{d})$. However, since differentiability of a sum $g+h$ does not imply differentiability of $g$, this is not enough to verify Theorem~\ref{thm:pu_tds}. Moreover, we caution that the typical behaviour in a shifted $\lip_{1}$ space can be very different to that in the natural space; Lemma~\ref{lemma:Fsig_unif_ndiff} of the present work may be used to produce examples demonstrating this. To verify that $P$ additionally captures points of differentiability of the typical function in $\lip_{1}([0,1]^{d})$, we adapt the construction of \cite{CPT_2005} so that the function $h$ is differentiable at almost all points of the set $P$. Differentiability of $g+h$ at such points then implies differentiability of $g$. 

The conclusions \eqref{non_sig_fin} and \eqref{all_proj} of Theorem~\ref{thm:pu_tds} come from the observation that the differentiability points of the typical $g\in\lip_{1}([0,1]^{d})$ inside of $P$ correspond to the differentiability points of the function $g+h$ inside the (necessarily much larger) universal differentiability set $\hl{Y}$. The latter set of points is large in the sense of \eqref{all_proj} due to \cite[Lemma~2.1]{dymond_maleva2016}. Although purely unrectifiable sets are regarded as completely opposite to universal differentiability sets, conclusions \eqref{non_sig_fin} and \eqref{all_proj} of Theorem~\ref{thm:pu_tds} show that they can be surprisingly close. For the typical function $f\in\lip_{1}([0,1]^{2})$ we find just as many points of differentiability of $f$ in the senses of \eqref{non_sig_fin} and \eqref{all_proj} inside the purely unrectifiable set $P$ as one might expect to find inside a universal differentiability set.

A further objective of this work is to provide a simplification of the argument in \cite{CPT_2005}, based on recent advances in the theories of universal differentiability and uniformly purely unrectifiable sets. There are two main tools in the simplification: Firstly, we make use of the recently announced result of M\'athe, that the notions of pure unrectifiability and uniform pure unrectifiability coincide. Since the condition for pure unrectifiability is significantly easier to verify, this immediately removes much of the complexity of the argument in \cite{CPT_2005}. The second main way in which we achieve a simplification is in a more special choice of the universal differentiability set $E$. We take $E$ as a universal differentiability set of the form described in \cite[Example~4.4]{maleva_preiss2018}: A $G_\delta$ set containing all lines from a dense subset of the set of all lines with directions inside a small cone. 
	
	Whilst we aspire to provide a more accessible proof of the result in \cite{CPT_2005}, we additionally obtain a stronger statement. We show that inside a universal differentiability set in $\R^{2}$ even directional derivatives of a Lipschitz function may be rather scarce. 
	\begin{theorem}\label{thm:main_result}
		For every $\alpha>0$ there exists a universal differentiability set $E\subseteq \R^{2}$ with the following property. There exists a Lipschitz function $h\colon \R^{2}\to\R$, and a double sided cone $\widehat{C}\subseteq S^{1}$ of width at most $\alpha$ such that the set of points in $E$ where $h$ has a directional derivative in any direction in $S^{1}\setminus \widehat{C}$ is contained in a purely unrectifiable set.
	\end{theorem}

\section{Preliminaries and Notation.}
We use the term~$\C^{1}$-curve to refer to a $\C^{1}$ mapping $\gamma$ from a closed interval $I\subseteq \R$ to $\R^{d}$ satisfying $\gamma'(t)\in S^{d-1}$ for all $t\in I$. Here $\gamma'(t)$ denotes the derivative of $\gamma$ at the point $t$ (or the one-sided derivative if $t$ is an endpoint). We identify this derivative with an element of $\R^{d}$ (or in this case $S^{d-1}$) in the standard way. A Borel set $P\subseteq \R^{d}$ is said to be \emph{purely unrectifiable} if for every $\C^{1}$-curve $\gamma\colon I\to\R^{d}$ the set $\gamma^{-1}(P)$ has Lebesgue measure zero. 

For $w\in S^{d-1}$ and $\alpha\in[0,1]$ we define a set
\begin{equation*}
C(w,\alpha):=\left\{v\in S^{d-1}\colon \langle{v,w}\rangle \geq 1-\alpha\right\},
\end{equation*} 
and refer to this set as the \emph{cone} around $w$ of width $\alpha$. We additionally define
\begin{equation*}
\widehat{C}(w,\alpha)=\left\{v\in S^{d-1}\colon \left|\langle{v,w}\rangle\right| \geq 1-\alpha\right\}
\end{equation*}
and call this set the \emph{double sided cone} around $w$ of width $\alpha$. Observe that $\widehat{C}(w,\alpha)=C(w,\alpha)\cup C(-w,\alpha)=C(w,\alpha)\cup-C(w,\alpha)$.

For a function $f\colon \R^{d}\to\R$ and $x\in\R^{d}$ we write $Df(x)$ for the derivative of $f$ at the point $x$ if it exists and we identify this with the unique element of $\R^{d}$ satisfying $Df(x)=\langle{Df(x),\cdot}\rangle$. To detect non-differentiability of $f$, we utilise the following test quantities. Given a point $z\in\R^{2}$ a direction $e\in S^{1}$ and $\varepsilon>0$ we consider the quantity
	\begin{equation}\label{def:zeta}
	\zeta(f,z,\varepsilon,e):=\sup\abs{\frac{f(x+te)-f(x)}{t}-\frac{f(y+se)-f(y)}{s}},
	\end{equation}
	where the supremum is taken over all segments of the form $[x,x+te]$ , $[y,y+se]$ satisfying $z\in [x,x+te]\cap[y,y+se]$ and $s,t\in[-\varepsilon,\varepsilon]\setminus\left\{0\right\}$. We further consider the related quantity $\Upsilon(f,z,\varepsilon)$ where the variable $e\in S^{1}$ is `moved inside the supremum', that is
	\begin{equation}\label{def:Upsilon}
	\Upsilon(f,z,\varepsilon):=\sup_{e\in S^{1}} \zeta(f,z,\varepsilon,e).
	\end{equation}
Roughly speaking, both quantities $\zeta(f,z,\varepsilon,e)$ and $\Upsilon(f,z,\varepsilon)$ reflect non-differentiability of $f$ at $z$ at scale $\varepsilon$. Severity of non-differentiability of $f$ at $z$ is sharply quantified by their limiting behaviour as $\varepsilon\to 0$. 
\begin{restatable}{proposition}{propzetacharactdiff}\label{prop:zeta_charact_diff} 
Let \hl{$A\subseteq \R^{d}$ be open}, $f,\hl{f_{1},f_{2}}\colon \hl{A}\to\R$ be Lipschitz functions, $z\in \hl{A}$ and $e\in S^{d-1}$. Then,
	\begin{enumerate}[(a)]
		\item\label{zeta_charact_diff} $\displaystyle\lim_{\varepsilon\to 0}\zeta(f,z,\varepsilon,e)=0 \quad \Leftrightarrow \quad$ $f$ has a directional derivative at $z$ in direction $e$. 
		\item\label{ups_non_diff}  $\displaystyle\limsup_{\varepsilon\to 0}\Upsilon(f,z,\varepsilon)>0\quad \Rightarrow \quad $ $\exists\, u\in S^{d-1}$ such that $\displaystyle\limsup_{\varepsilon\to 0}\zeta(f,z,\varepsilon,u)>0$.
		\hl{\item\label{mono} The function $(0,\infty)\to(0,\infty)$, $\varepsilon\mapsto \Upsilon(f,z,\varepsilon)$ is increasing.}
		\hl{\item \label{sum_diff_sup_ineq} 
		$\Upsilon(f_{1}+f_{2},z,\varepsilon)\geq \Upsilon(f_{1},z,\varepsilon)-\Upsilon(f_{2},z,\varepsilon)$.}
	\end{enumerate}	
\end{restatable}
The proof of Proposition~\ref{prop:zeta_charact_diff} is a standard exercise \hl{in differentiability and the definitions \eqref{def:Upsilon} and \eqref{def:zeta}}. 
%We postpone it until the Appendix. 
The next lemma plays a key part in the proof of Theorem~\ref{thm:pu_tds}. It allows us to preserve non-differentiability of a Lipschitz function $h$ at many points after adding a typical function~$g$.

\hl{For the proof of Lemma~\ref{lemma:Fsig_unif_ndiff} we will require a version of the Banach-Mazur game, described in \cite[Section~8.H]{kechris2012classical}. We recall the details here: 
\paragraph{\hl{The Banach-Mazur game $\BM(A,X)$:}}	
Let $X$ be a non-empty topological space and $A$ be a subset of $X$. Two players, Player~I and Player~II, take it in turns to choose non-empty, open subsets of $X$, denoted by $U_{n}$ and $V_{n}$. Player~I begins the game by choosing the set $U_{1}\subseteq X$ and then Player~II responds by choosing $V_{1}\subseteq U_{1}$. Then Player~I chooses $U_{2}\subseteq V_{1}$ and Player~II chooses $V_{2}\subseteq U_{2}$ and so on. Thus, the game produces a sequence of non-empty, open sets
	\begin{equation*}
	X\supseteq U_{1}\supseteq V_{1}\supseteq U_{2}\supseteq V_{2}\supseteq \ldots,
	\end{equation*}
	where for each $n\in\N$ the set $U_{n}$ is referred to as the $n$-th move of Player~I and the set $V_{n}$ as the $n$-th move of Player~II. We say that Player~II wins the game if $\bigcap_{n=1}^{\infty}V_{n}\subseteq A$, or equivalently, if $\bigcap_{n=1}^{\infty}U_{n}\subseteq A$. Otherwise Player~I wins.

The important fact about the Banach-Mazur game that we will require is it that it can be used to characterise residual sets. More precisely, a subset $A$ of a non-empty topological space $X$ is residual if and only if Player~II has a winning strategy in the Banach-Mazur game $\BM(A,X)$, \cite[Thm~8.33]{kechris2012classical}.}

\begin{lemma}\label{lemma:Fsig_unif_ndiff}
	Let $K\subseteq \hl{(}0,1\hl{)}^{d}$ be an $F_{\sigma}$ set, $\sigma>0$ and $h\colon[0,1]^{d}\to\R$ be a Lipschitz function satisfying
	\begin{equation}\label{eq:non_diff_h}
	\limsup_{\varepsilon\to 0}\Upsilon(h,z,\varepsilon)\geq \sigma.
	\end{equation}
	for all $z\in K$. Then for typical $g\in\lip_{1}([0,1]^{d})$ we have
	\begin{equation}\label{eq:bm_target}
	\limsup_{\varepsilon\to 0}\Upsilon(h+g,z,\varepsilon)\geq \sigma
	\end{equation}	
	for all $z\in K$.
\end{lemma}
\begin{proof}
	We may assume that $K$ is compact. Let $0<\lambda<\lambda'<\lambda''<\sigma$. It suffices to \hl{verify \eqref{eq:bm_target} with $\sigma$ replaced by $\lambda$ for} the typical $g\in\lip_{1}([0,1]^{d})$. We describe a winning strategy for Player II in the relevant Banach-Mazur game
	\hl{\begin{equation*}
	\BM\left(\set{g\in\lip_{1}([0,1]^{d})\colon \limsup_{\varepsilon\to 0}\Upsilon(h+g,z,\varepsilon)\geq\lambda},\lip_{1}([0,1]^{d})\right),
	\end{equation*}
defined before the present lemma and in \cite[Section~8.H]{kechris2012classical}. To complete the proof, it then only remains to apply \cite[Thm~8.33]{kechris2012classical}.}	
	
In response to the \hl{non-empty, open subset $U_{n}$ of $\lip_{1}([0,1]^{d})$ chosen as the $n$-th move} of Player I, Player II chooses a smooth function $g_{n}\in U_{n}$ and $\delta_{n}\in(0,2^{-n})$ so that $B(g_{n},\delta_{n})\subseteq U_{n}$. Next, Player II chooses for each point $z\in K$ a direction $e(z)\in S^{d-1}$, points $x(z),y(z)\in[0,1]^{d}$ and numbers $s(z),t(z)\in[-\delta_{n},\delta_{n}]\setminus\left\{0\right\}$ witnessing, \hl{according to \eqref{def:Upsilon} and \eqref{def:zeta}}, that
	\begin{equation}\label{eq:Ups_lambda''}
	\hl{\Upsilon}(\hl{h+g_{n}},\hl{z},\delta_{n})>\lambda''.
	\end{equation}
	\hl{The latter inequality~\eqref{eq:Ups_lambda''} holds for all $z\in K$ due to the smoothnees of $g_{n}$ and \eqref{eq:non_diff_h}. More precisely, the smoothness of $g_{n}$ in combination with Proposition~\ref{prop:zeta_charact_diff}~\eqref{zeta_charact_diff} and \eqref{ups_non_diff} implies that $\limsup_{\varepsilon\to 0}\Upsilon(g_{n},z,\varepsilon)=0$ for all $z\in (0,1)^{d}$. Putting this together with \eqref{eq:non_diff_h} and Proposition~\ref{prop:zeta_charact_diff}~\eqref{sum_diff_sup_ineq}, we deduce that $\limsup_{\varepsilon\to 0}\Upsilon(h+g_{n},z,\varepsilon)\geq \sigma>\lambda''$ for all $z\in K$. Finally, we apply Proposition~\ref{prop:zeta_charact_diff}~\eqref{mono}, to obtain \eqref{eq:Ups_lambda''}.} 
	
	Given $z'\in B(z,\eta)$ we have for $w(z'):=z'-z$ that $z'\in[x+w,x+w+te]\cap[y+w,y+w+se]$ and 
	\begin{multline*}
	\left|\frac{(h+g_{n})(x+w+te)-(h+g_{n})(x+w)}{t}-\frac{(h+g_{n})(y+w+se)-(h+g_{n})(y+w)}{s}\right|\\ \geq \lambda''-\frac{4(\lip(h)+1)\eta}{\min\left\{s,t\right\}}.
	\end{multline*}
	Let now $\eta(z)$ be sufficiently small so that 
	\begin{multline}
	\left|\frac{(h+g_{n})(x+w+te)-(h+g_{n})(x+w)}{t}-\frac{(h+g_{n})(y+w+se)-(h+g_{n})(y+w)}{s}\right|\\
	>\lambda' \label{eq:ball_ineq}
	\end{multline}
	for all points $z'\in B(z,\eta(z))$. The collection $(B(z,\eta(z)))_{z\in K}$ is an open cover of the compact set $K$. Player II extracts a finite subcover $(B(z_{i},\eta(z_{i})))_{i=1}^{N}$ and returns the open set $V_{n}:=B(g_{n},\theta_{n})$ for $\theta_{n}$ chosen sufficiently small based on the data corresponding to the points $z_{1},\ldots,z_{N}$ and, in particular, small enough so that $V_{n}\subseteq U_{n}$. The precise remaining condition on $\theta_{n}$ that we require will be determined later in the proof.
	
	Let us now verify that Player II wins the Banach Mazur game following the above strategy. Let $g\in \bigcap_{n=1}^{\infty}V_{n}$ and $z\in K$. We need to prove $\limsup_{\varepsilon\to 0}\Upsilon(h+g,z,\varepsilon)\geq \lambda$. Fixing $\varepsilon>0$ we verify that $\Upsilon(h+g,z,\varepsilon)\geq \lambda$. Let $n\in\N$ be large enough so that $\delta_{n}<\varepsilon$ and let $z_{i}$ be one of the points corresponding to the ball $B(z_{i},\eta(z_{i}))$ chosen by Player~II in the $n$-th round of the Banach-Mazur game such that $z\in B(z_{i},\eta(z_{i}))$. Then for $w:=z-z_{i}$ and $(x,y,s,t)=(x(z_{i}),y(z_{i}),s(z_{i}),t(z_{i}))$ we have that 
	\begin{equation*}
	 z\in[x+w,x+w+te]\cap [y+w,y+w+se],,\qquad s,t\in[-\delta_{n},\delta_{n}]\setminus\left\{0\right\}\subseteq[-\varepsilon,\varepsilon]\setminus\set{0}
	\end{equation*}
	and that $\eqref{eq:ball_ineq}$ holds. Since $g\in B(g_{n},\theta_{n})$, the same inequality holds with \hl{$g_{n}$} replaced by \hl{$g$} and $\lambda'$ replaced by $\lambda'-\frac{4\theta_{n}}{\min\left\{s,t\right\}}$. Thus, we obtain $\zeta(h+g,z,\varepsilon,\hl{e})\geq \lambda$ with the condition 
	\begin{equation*}
	\theta_{n}\leq \frac{(\lambda'-\lambda)\min_{1\leq i\leq N}\left\{s(z_{i}),t(z_{i})\right\}}{4}
	\end{equation*}
	imposed on \hl{$\theta_{n}$}. Here the \hl{minimum} is taken over all points $z_{1},\ldots,z_{N}\in K$ chosen by Player~II in the $n$-th round of the game.
	\end{proof}

\section{Construction of a Universal Differentiability Set.}
	We present a construction of a universal differentiability set $E\subseteq \R^{2}$ and a Lipschitz function $h$ having very few differentiability points in $E$. This will serve both the proof of Theorem~\ref{thm:main_result} and the proof of Theorem~\ref{thm:pu_tds}. The construction is \hl{primarily} based on that of \cite{CPT_2005}, but contains a few new modifications. Crucially for the proof of Theorem~\ref{thm:pu_tds}, we modify the construction in order to distinguish points of the set $E$ where $h$  is differentiable.
	\subsection{The Set $E$.}\label{section:set}
	Let $E\subseteq\R^{2}$ be a set of the form
	\begin{equation}\label{eq:set_E}
	E=\bigcap_{n=1}^{\infty}\bigcup_{k=n}^{\infty}B(L_{k},\rho_{k}),
	\end{equation}
	where $(L_{k})_{k=1}^{\infty}$ is a sequence of lines $L_{k}\subseteq \R^{2}$ and $(\rho_{k})_{k=1}^{\infty}$ is a sequence of positive numbers $\rho_{k}$ which converges to zero sufficiently \hl{rapidly}, in particular so that
	\hl{\begin{equation*}
	\sum_{k=1}^{\infty}\rho_{k}<\infty.
	\end{equation*}}\hl{Precisely six} further conditions will be imposed on these sequences in the course of the proof. \hl{To help the reader keep track of all of these conditions and verify their compatibility, we will use the labels \eqref{E1}, \eqref{E2}, \eqref{E3}, \eqref{E4}, \eqref{E5}, \eqref{E6} to mark each condition. We emphasise that in theory it is possible to state all of these conditions here immediately. However, by imposing them only at the moment that they are needed we hope to somewhat disentangle the proof and expose more clearly the purpose of each condition. In line with this convention, the statements of all lemmas which follow should be interpreted as being valid subject to additional conditions which may be imposed on the parameters of the construction in their proofs.}
	
	We define a sequence of functions $(k_{p})_{p=0}^{\infty}$ on $\R^{2}$ whose purpose is to record for each point $z\in\R^{2}$ the possibly empty subsequence of $\hl{i}\in\N$ for which $z\in B(L_{\hl{i}},\rho_{\hl{i}})$. Setting $k_{0}=0$ on the whole plane $\R^{2}$ we define $k_{p}$ inductively by
	\begin{equation}\label{eq:def_kp}
	k_{p}(z)=\inf\left\{\hl{i}>k_{p-1}(z)\colon z\in B(L_{\hl{i}},\rho_{\hl{i}})\right\},
	\end{equation}
where we interpret the infimum of the empty set as $\infty$. 
\hl{\begin{enumerate}
	\item[(\mylabel{E1}{E1})] We impose an additional constraint on the set $E$, namely, that the directions $e_{k}$ of each line $L_{k}$ lie in a cone around a fixed vector $w\in S^{1}$. For a parameter $\eta\in(0,1]$ we demand that
	\begin{equation}\label{eq:line_dir_cone}
	e_{k}\in C(w,\eta) \qquad\text{for all $k\in\N$.}
	\end{equation}
	The parameter $\eta\in(0,1)$ should be assumed to be small; in what follows we will occasionally require that it is smaller than some absolute constant whose value is not important. Eventually, for the proof of Theorem~\ref{thm:main_result}, the sufficiently small condition on $\eta$ will be determined by $\alpha$.
\end{enumerate}}
For each line $L_{k}$ we fix a point $x_{k}\in L_{k}$ so that $L_{k}=x_{k}+\R e_{k}$. We can now formulate a sufficient condition for $E$ to be a universal differentiability set.
\begin{lemma}\label{prop:suff_uds}
	\hl{Suppose that the sequence of lines $(L_{k}=x_{k}+\R e_{k})_{k=1}^{\infty}$ is such that the sequence of pairs $((x_{k},e_{k}))_{k=1}^{\infty}$ is dense in $\R^{2}\times C(w,\eta)$. Then,
		\begin{enumerate}[(i)]
			\item\label{l:E_UDS_i} the set $E$ is a universal differentiability set.
			\item\label{l:E_UDS_ii} there exists a (possibly different) sequence of lines $(\widetilde{L}_{k}=\widetilde{x}_{k}+\R\widetilde{e}_{k})_{k=1}^{\infty}$ for which the sequence of pairs $((\widetilde{x}_{k},\widetilde{e}_{k}))_{k=1}^{\infty}$ is dense in $\R^{2}\times C(w,\eta)$ and $\widetilde{L}_{k}\subseteq E$ for all $k\in\N$.
		\end{enumerate}
	}
\end{lemma}
\begin{proof}
	For \hl{\eqref{l:E_UDS_i}} see \cite[Example~4.4]{maleva_preiss2018}. \hl{\eqref{l:E_UDS_ii}} is proved by a Baire Category argument given in \cite[p.~362]{CPT_2005}.
\end{proof}
\hl{\begin{enumerate}
	\item[(\mylabel{E2}{E2})] We demand that the sequence $(L_{k})_{k=1}^{\infty}$ of lines satisfies the condition of Lemma~\ref{prop:suff_uds}, so that $E$ is a universal differentiability set.
\end{enumerate}} 
The next lemma represents a key step in the proof of Theorem~\ref{thm:pu_tds}. It is not needed for the proof of Theorem~\ref{thm:main_result}.
\begin{lemma}\label{lemma:compact_uds}
	There is a compact universal differentiability set $\hl{Y}\subseteq E\cap[0,1]^{2}$. 
\end{lemma}
\begin{remark}
	In \cite{Dore_Maleva2}, Dor\'e and Maleva give a construction of a compact universal differentiability set inside a given $G_{\delta}$ set containing a sequence of lines dense in $\R^{d}\times S^{d-1}$ in the sense of Lemma~\ref{prop:suff_uds}. The proof of Lemma~\ref{lemma:compact_uds}, where inside the $G_{\delta}$ set $E$ we only have density of lines in $\R^{2}\times C(w,\eta)$, requires several simple modifications to this construction and to arguments presented in the preceding paper \cite{Dore_Maleva1} of the same authors. These arguments have also been employed in subsequent works \cite{Dore_Maleva3} and \cite{dymond_maleva2016}. Since the full details of the modification would be rather lengthy, we present below a sketch of the proof of Lemma~\ref{lemma:compact_uds} which refers to the relevant literature and describes the necessary modifications.  
\end{remark}
\begin{proof}[Proof of Lemma~\ref{lemma:compact_uds}]
	The $G_{\delta}$ set $E$ contains a sequence of lines $(\widetilde{L}_{k}=\widetilde{x}_{k}+\R\widetilde{e}_{k})_{k=1}^{\infty}$ which is dense in $\R^{2}\times C(w,\eta)$ in the sense of Lemma~\ref{prop:suff_uds}. We follow the construction of \cite{Dore_Maleva2} to produce a family of compact sets inside of $E$. The construction provides families of sets of the form
	\begin{equation*}
	M_{k}(\lambda)=\bigcup_{k\leq n\leq (1+\lambda)k}\overline{B}_{\lambda w_{n}}(R_{n})\subseteq[0,1]^{2},\qquad \lambda\in (0,1],
	\end{equation*}
	where the sets $R_{n}$ are increasing, finite unions of line segments \hl{and the numbers $w_{n}>0$ are chosen sufficiently small}. In our modified construction the line segments of $R_{n}$ will always be chosen inside the lines $\widetilde{L_{k}}\subseteq E$. The universal differentiability sets produced by \cite{Dore_Maleva2} take the form
	\begin{equation*}
	T_{\lambda}=\bigcap_{k=1}^{\infty}M_{k}(\lambda)\subseteq [0,1]^{2},\qquad \lambda\in (0,1],
	\end{equation*}
	and the construction ensures that each set $T_{\lambda}$ fits inside a $G_{\delta}$ set fixed at the start containing all lines added to the sets $R_{n}$. We take $E$ as this $G_{\delta}$ set and so we obtain compact sets $T_{\lambda}\subseteq E$. Further note that the sets $(T_{\lambda})_{\lambda\in (0,1]}$ are nested in the sense that $T_{\lambda_{1}}\subseteq T_{\lambda_{2}}$ whenever $\lambda_{1}\leq \lambda_{2}$. 
		
	To establish that each of the sets $T_{\lambda}$ is a universal differentiability set, the paper \cite{Dore_Maleva2} proves that the family $(T_{\lambda})_{\lambda\in(0,1]}$ posseses the `wedge approximation property' described in \cite[Lemma~3.5]{Dore_Maleva3} and \cite[Lemma~3.1]{dymond_maleva2016}. In our modified construction, we only add line segments to the sets $R_{n}$ with directions inside the cone $C(w,\eta)$. Accordingly, we obtain sets $(T_{\lambda})_{\lambda\in (0,1]}$ with a weaker form of the wedge approximation property. Namely, the identical approximation property restricted only to wedges $[x,y]\cup[y,z]$ in which the two line segments $[x,y]$ and $[y,z]$ are both parallel to some direction in the cone $C(w,\eta/2)$.  We write $\eta/2$ instead of $\eta$ here to avoid problems with directions on the boundary.
	
	It now remains to argue that this restricted wedge approximation property is sufficient for universal differentiability. Given a Lipschitz function $f_{0}\colon\R^{2}\to\R$ we follow the proof of \cite[Theorem~3.1]{Dore_Maleva1} in order to find a point of differentiability of $f_{0}$ inside say $T_{1}$. To begin, we fix some $\lambda_{0}<\lambda_{1}\in (0,1)$ and find a pair $(x_{0},e_{0})$ with $x_{0}\in T_{\lambda_{0}}$ and $e_{0}\in S^{1}$ such that the directional derivative $f'(x_{0},e_{0})$ exists. Since the set $T_{\lambda_{0}}$ contains line segments in $R_{1}$, parallel to some direction in the cone $C(w,\eta/2)$ we may additionally prescribe here that the direction $e_{0}$ is taken inside $C(w,\eta/2)$. Given this starting data, the proof of \cite[Theorem~3.1]{Dore_Maleva1} constructs a Lipschitz function $f\colon\R^{2}\to\R$ which differs from $f_{0}$ only by a linear function and a sequence of point-direction pairs $(x_{n},e_{n})\in T_{\lambda_{1}}\times S^{1}$ converging to a pair $(x,e)\in T_{\lambda_{1}}\times S^{1}$ such that the directional derivative $f'(x,e)$ exists and satisfies a very delicate `almost locally maximal' condition defined in the statement of \cite[Theorem~3.1]{Dore_Maleva1}. In the iterative construction of the sequence $(x_{n},e_{n})$ the new direction $e_{n+1}$ may always be chosen arbitrarily close to the previous one $e_{n}$; in this proof the inequality $\norm{e_{n+1}-e_{n}}\leq \sigma_{n}$ is satisfied at each step where $\sigma_{n}$ may always be taken arbitrarily small. Hence, we may ensure that the limit direction $e$ lies inside $C(w,\eta/2)$. 
	
	Finally, having arrived at a pair $(x,e)\in T_{\lambda_{1}}\times C(w,\eta/2)$ for which the directional derivative $f'(x,e)$ is almost locally maximal, we argue that $f$ and therefore also $f_{0}$ is differentiable at $x$. In what follows the point $z\in\R^{2}$ is denoted by $\lambda$ in the referred literature. We change the notation in this instance in order to avoid confusion with the index $\lambda$ of the sets $T_{\lambda}$, but otherwise we use the same notation as the referred literature. If $f$ is not differentiable at $x$ then we follow the argument of \cite[Lemma~4.3]{Dore_Maleva1} and use \cite[Lemma~4.2]{Dore_Maleva1} to show that on arbitrarily small wedges of the form
	\begin{equation}\label{eq:wedge}
	[x-se,x+z]\cup[x+z,x+se]\subseteq \R^{2},
	\end{equation}
	and on all sufficiently good approximations of such wedges we may find points $x'$ admitting a direction $e'$ for which the directional derivative $f'(x',e')$ exists and is greater, in a technical sense, than $f'(x,e)$. If such wedges can be found inside the sets $T_{\alpha}$ with $\alpha$ greater than but arbitrarily close to $\lambda_{1}$ then we obtain a contradiction to the almost locally maximal condition on $f'(x,e)$, which completes the proof. In \cite{Dore_Maleva2} this is ensured by the wedge approximation property of the sets $(T_{\lambda})_{\lambda\in(0,1]}$. The point $x+z$ appearing in \eqref{eq:wedge} may be taken arbitrarily close to the line segment $[x-se,x+se]$ relative to the scale $s>0$; see \cite[(4.4), Lemma~4.2]{Dore_Maleva1}. Therefore, the directions of the two segments $[x-se,x+z]$ and $[x+z,x+se]$ may be taken arbitrarily close to $e\in C(w,\eta/2)$. In particular, it suffices to consider only wedges in which the two component line segments are parallel to directions in $C(w,\eta/2)$. This means that the restricted wedge approximation property present in our sets $(T_{\lambda})_{\lambda\in (0,1]}$ is enough.	
\end{proof}
\begin{remark}
	The argument used in the proof of Lemma~\ref{lemma:compact_uds} also shows that there exist compact universal differentiability sets of arbitrarily small cone width in the sense of \cite[Definition~1.1]{maleva_preiss2018}. 
\end{remark}
	\subsection{$\C^{1}$ curves meeting $E$.}
	Recall that our ultimate goal is to construct a function $h\colon \R^{2}\to \R$ whose set of differentiability points inside of $E$ intersect every $\C^{1}$ curve in a set of measure zero. The objective of the present section is to investigate how $C^{1}$ curves intersect the whole set $E$. The results that follow depend entirely on the geometry of the set $E$ and in particular rely on the thinness of the strips $B(L_{k},\rho_{k})$. They have nothing to do with the function with a small set of differentiability points that we will construct later on. 
	\begin{lemma}\label{lemma:easy_curves}
	For every $\C^{1}$ curve $\gamma\colon I\to\R^{2}$ satisfying
	\begin{equation*}
	\gamma'(t)\notin \widehat{C}(w,\eta) \qquad \text{for all $t\in I$,}
	\end{equation*}
	it holds that $\leb\left(\gamma^{-1}(E)\right)=0$.
	\end{lemma}
	\begin{proof}
%		We may assume that $\langle{\gamma'(t),w}\rangle$ does not change sign for $t\in I$. Otherwise, split $\gamma$ into finitely many subcurves with this property and apply the following argument (up until a change of sign) to each one separately. Without loss of generality, we assume $\langle{\gamma'(t),v}\rangle\geq 0$ for all $t\in I$. 
		The function
		\begin{equation*}
		I\times C(w,\eta)\to \R,\qquad (t,e)\mapsto \abs{\langle{\gamma'(t),e}\rangle},
		\end{equation*}
		is continuous and defined on a compact set. Therefore, it attains its maximum, which must be greater than zero, at some pair $(t_{0},e_{0})\in I\times C(w,\eta)$. Since $\gamma'(t_{0})\notin \widehat{C}(w,\eta)$ and $e_{0}\in C(w,\eta)$ we have 
		\begin{equation*}
		0<\abs{\langle{\gamma'(t_{0}),e_{0}}\rangle} <1.
		\end{equation*}
		Setting $\delta_{0}:=1-\hl{\sqrt{1-\abs{\langle{\gamma'(t_{0}),e_{0}}\rangle}^{2}}}$, we deduce, \hl{using the maximality of $\abs{\langle{\gamma'(t_{0}),e_{0}}\rangle}$}, that
		\begin{equation*}
		\gamma'(t)\in \widehat{C}(e^{\perp},\delta_{0})\qquad \text{for all $t\in I$,}
		\end{equation*}
		for all $e\in C(w,\eta)$ and in particular for all $e=e_{k}$, $k\in\N$. \hl{Recalling that $B(L_{k},\rho_{k})$ is a strip of width $2\rho_{k}$ parallel to $e_{k}$,} elementary geometric reasoning leads to 
		\begin{align*}
		\gamma^{-1}(B(L_{k},\rho_{k}))&\leq \frac{2\rho_{k}}{1-\delta_{0}},
		\end{align*}
		\hl{for each $k\in\N$.} More precisely, we obtain the above inequality by applying Lemma~\ref{lemma:crossing} of Appendix~\ref{app_geom_curves} with $W=B(L_{k},\rho_{k})$, $v=e_{k}^{\perp}$ and $\delta=\delta_{0}$. Since, for arbitrary $N\in\N$ the set $\bigcup_{k=N}^{\infty}\overline{B}(L_{k},\rho_{k})$ covers $E$ \hl{(see \eqref{eq:set_E})}, we have
		\begin{equation*}
		\leb(\gamma^{-1}(E))\leq \frac{2}{1-\delta_{0}}\sum_{k=N}^{\infty}\rho_{k}\qquad \text{for all $N\in\N$,}
		\end{equation*}
		and hence $\leb(\gamma^{-1}(E))=0$.
	\end{proof}
	The remaining results of the present section share a common hypothesis. \hl{Before stating it, we will try to provide some intuition. For the Lipschitz function $h\colon\R^{2}\to\R$ that we construct later, we will need to show that $\Diff(h)\cap E$ meets every $\C^{1}$ curve in a set of Lebesgue measure zero. Since every $\C^{1}$ curve may be partitioned into shorter $\C^{1}$ curves, whose derivatives are almost constant, it suffices to consider only curves whose derivative stays inside a cone of arbitrarily thin width. If such a cone is taken away from $w$ or $-w$ then the situation is easy: Lemma~\ref{lemma:easy_curves} establishes that the entire set $E$ is invisible to curves corresponding to such a cone. The following hypothesis considers the problematic case of curves which are almost parallel to $w$, that is, those curves whose derivatives stay inside a thin cone with centre $w$. For such curves we require some additional work to show that they intersect $\Diff(h)\cap E$ in a set of Lebesgue measure zero. To achieve this we will approximate their derivatives by simpler mappings, denoted by $\beta_{p}$ in Hypothesis~\ref{hyp} below.}
	\begin{hypothesis}\label{hyp}
		Let $\delta\in(0,1)$ be sufficiently small, \hl{that is, smaller than some positive, absolute constant whose value is not important}\footnote{\hl{The precise `sufficiently small condition' on $\delta$ is determined by \eqref{eq:suff_small_delta} inside the proof of Lemma~\ref{lemma:Dp_integral}.}}. Let $\gamma\colon I\to\R^{2}$ be a $\C^{1}$ curve and suppose that
		\begin{equation*}
		\gamma'(t)\in C(w,\delta) \qquad\text{for all $t\in I$.}
		\end{equation*}  
		For each $p\geq 0$ let $\Sigma_{p}$ be the smallest $\sigma$-algebra on $I$ with respect to which the functions 
		\begin{equation*}
		k_{q}\circ \gamma, \qquad q=0,1,2,\ldots,p
		\end{equation*}
		are measurable. \hl{(See \eqref{eq:def_kp} for the definition of the functions $k_{q}$).} Furthermore, we define for each $p\geq 0$ a mapping $\beta_{p}\colon I\to\R^{2}$ by 
		$\beta_{p}=\Exp[\gamma'|\Sigma_{p}]$ and consider the corresponding sets
		\begin{align}\label{eq:Dp}
		D_{p}:=\left\{t\in I\colon k_{p}(\gamma(t))<\infty,\, \abs{\skp{\beta_{p}(t),e_{k_{p}(\gamma(t))}^{\perp}}}>2^{-p}\right\},\quad
		D:=\bigcap_{n=1}^{\infty}\bigcup_{p=n}^{\infty}D_{p}.	
		\end{align}
	\end{hypothesis}

	\begin{lemma}[Under Hypothesis~\ref{hyp}]\label{prop:Dp}
The set $D\subseteq I$ has Lebesgue measure zero.
	\end{lemma}

The proof of Lemma~\ref{prop:Dp} is based on the following observation:
\begin{lemma}[Under Hypothesis~\ref{hyp}]\label{lemma:Dp_integral}
	Let $k,\hl{p}\in \N$ and $P$ be a \hl{connected} component of 
	\begin{equation*}
	B(L_{k},\rho_{k})\setminus \bigcup_{1\leq j< k}\partial B(L_{j},\rho_{j})
	\end{equation*}
	for which $k_{p}(z)=k$ for all $z\in P$. Then 
	\begin{equation*}
	\int_{\gamma^{-1}(P)}\left|\langle{\beta_{p}(t),e_{k}^{\perp}}\rangle\right|\,dt\leq 12\rho_{k}.
	\end{equation*}

\end{lemma}
\begin{proof}
%	We may assume that $\gamma$ is $C^{1}$. Then, for a fixed $c\in(0,1)$, $\gamma$ may be decomposed into finitely many subcurves satisfying either
%	\begin{equation*}
%	\left|\langle{\gamma'(t),e_{k}}\rangle\right|\geq c\quad\text{for all $t$,}\quad\text{or, }\left|\langle{\gamma'(t),e_{k}}\rangle\right|<c\quad\text{for all $t$.}
%	\end{equation*}
%	We apply Lemma~\ref{lemma:convex_curve} with $e=\pm e_{k}$ to those curves $\gamma$ satisfying the first inequality, obtaining
%	\begin{equation*}
%	\int_{\gamma^{-1}(P)}\langle{\gamma'(t),e_{k}^{\perp}}\rangle\,dt\leq 10\rho_{k}.
%	\end{equation*}
%	We distinguish two cases. First assume that $\left|\langle{v,e_{k}}\rangle\right|\leq \frac{1}{\sqrt{2}}$. Then for all $t\in I$ an elementary computation gives
%	\begin{equation*}
%	\abs{\langle{\gamma'(t),e_{k}^{\perp}}\rangle}\geq (1-\delta)\frac{1}{\sqrt{2}}-\sqrt{\delta(2-\delta)}\frac{1}{\sqrt{2}}\geq \frac{1}{2\sqrt{2}},
%	\end{equation*}
%	where the last inequality holds for $\delta$ sufficiently small. Then, by Lemma~\ref{lemma:crossing}, we have $\leb(\gamma^{-1}(P))\leq 4\sqrt{2}\rho_{k}$, from which the result follows.
	
Note that $\langle{w,e_{k}}\rangle \geq 1-\eta\geq \frac{1}{\sqrt{2}}$, where the final inequality is a condition on $\eta$. Thus, for all $t\in I$ we have
	\begin{equation}\label{eq:suff_small_delta}
	\langle{\gamma'(t),e_{k}}\rangle \geq (1-\delta)\frac{1}{\sqrt{2}}-\sqrt{\delta(2-\delta)}\frac{1}{\sqrt{2}}\geq \frac{1}{2\sqrt{2}},
	\end{equation}
	\hl{where the last inequality is the `sufficiently small condition' on $\delta$ referred to in Hypothesis~\ref{hyp}.} Hence, viewing $\R^{2}$ with the coordinate system $(e_{k},e_{k}^{\perp})$, $\gamma$ is a curve which moves strictly from left to right. Moreover, $P$ is an open, convex set \hl{given by a finite intersection of open half-spaces} and is contained in the horizontal strip $B(L_{k},\rho_{k})$ of width $2\rho_{k}$. These considerations imply a bound of order $\rho_{k}$ on the signed variation of the second coordinate function of $\gamma$ inside the set $P$. More precisely, by a geometric argument of \cite{CPT_2005}, extracted in Lemma~\ref{lemma:convex_curve} of Appendix~\ref{app_geom_curves}, we derive
	\begin{equation*}
	\left|\int_{\gamma^{-1}(P)}\langle{\gamma'(t),e_{k}^{\perp}}\rangle\,dt\right| \leq 12\rho_{k}.
	\end{equation*}	
To complete the proof we show that 
%the quantity on the left hand side above is equal to 
\hl{\begin{equation}\label{eq:eqn_to_prove}
\left|\int_{\gamma^{-1}(P)}\langle{\gamma'(t),e_{k}^{\perp}}\rangle\,dt\right|=\int_{\gamma^{-1}(P)}\abs{\hl{\langle{\beta_{p}(t),e_{k}^{\perp}\rangle}}}\,dt. 
\end{equation}}
For any fixed $z_{0}\in P$ the set $P$ \hl{satisfies}
\hl{\begin{equation*}
P\subseteq\bigcap_{q=0}^{p}\left\{z\in\R^{2}\colon k_{q}(z)=k_{q}(z_{0})\right\}=:A_{p}(z_{0}).
\end{equation*}}
\hl{To see this, fix $z_{0}\in P$ and $z\notin A:=A_{p}(z_{0})$. We verify that $z\notin P$, which will prove the assertion. Let $q\in\set{1,\ldots,p}$ be minimal such that $k_{q}(z)\neq k_{q}(z_{0})$. Then $z\in B(L_{k_{r}(z_{0})},\rho_{k_{r}(z_{0})})$ for $1\leq r<q$ and $z\notin B(L_{k_{q}(z_{0})},\rho_{k_{q}(z_{0})})$. Therefore one of the boundary lines of $B(L_{k_{q}(z_{0})},\rho_{k_{q}(z_{0})})$ separates $z$ from $z_{0}$ and accordingly $z$ and $z_{0}$ cannot belong to the same connected component of $B(L_{k_{p}(z_{0})},\rho_{k_{p}(z_{0})})\setminus \bigcup_{1\leq j<k_{p}(z_{0})}\partial B(L_{j},\rho_{j})$. Hence $z$ does not belong to $P$. 
	
We now have everything in place to verify \eqref{eq:eqn_to_prove}: For $A:=A_{p}(z_{0})$, we have that} $\gamma^{-1}(\hl{A})\in\Sigma_{p}$, all functions $k_{q}\circ \gamma$, $0\leq q\leq p$ are constant on $\gamma^{-1}(\hl{A})$ and $\beta_{p}$ is also constant on \hl{$\gamma^{-1}(A)\supseteq \gamma^{-1}(P)$}. It follows that
	\begin{equation*}
	\int_{\gamma^{-1}(P)}\left|\langle{\beta_{p}(t),e_{k}^{\perp}}\rangle\right|\,dt=\left|
	\int_{\gamma^{-1}(P)}\langle{\beta_{p}(t),e_{k}^{\perp}}\rangle\,dt\right|=\left|
	\int_{\gamma^{-1}(P)}\langle{\gamma'(t),e_{k}^{\perp}}\rangle\,dt\right|,
	\end{equation*}
	\hl{which delivers \eqref{eq:eqn_to_prove}.}
\end{proof}

We are now ready to give the proof of Lemma~\ref{prop:Dp}:
\begin{proof}[Proof of Lemma~\ref{prop:Dp}]
	It suffices to prove that the sequence $(\leb(D_{p}))_{p=1}^{\infty}$ is summable. The set $D_{p}$ can be expressed as the union of all sets 
	\begin{equation*}
	D_{p,k}:=\left\{t\in[0,1]\colon k_{p}(\gamma(t))=k,\quad \left|\langle{\beta_{p}(t),e_{k}^{\perp}}\rangle\right|>2^{-p}\right\}
	\end{equation*}
	for $k\geq p$. We observe that
	\begin{equation*}
	D_{p,k}\subseteq \bigcup \gamma^{-1}(P)
	\end{equation*}
	where the union is taken over all \hl{connected} components $P$ of $B(L_{k},\rho_{k})\setminus \bigcup_{1\leq j<k} B(L_{j},\rho_{k})$ for which $k_{p}(z)=k$ for all $z\in P$. Using the bound given by Lemma~\ref{lemma:Dp_integral} and the fact that there are at most $3^{k}$ such \hl{connected} components $P$ we deduce
	\begin{equation*}
	\int_{D_{p,k}}\left|\langle{\beta_{p}(t),e_{k}^{\perp}}\rangle\right|\,dt\leq 3^{k}\cdot 12\rho_{k}.
	\end{equation*}
	Summing this inequality over $k\geq p$ we obtain
	\begin{equation}\label{eq:expectation}
	\int_{D_{p}} \left|\langle{\beta_{p}(t),e_{k_{p}(\gamma(t))}^{\perp}}\rangle\right|\,dt \leq 12\sum_{k=p}^{\infty}3^{k}\rho_{k}\leq4^{-p}
	\end{equation}
	\hl{\begin{itemize}
		\item[(\mylabel{E3}{E3})] The last inequality, which may be written equivalently as $\sum_{k=p}^{\infty}3^{k}\rho_{k}\leq \frac{4^{-p}}{12}$, is a further condition that we impose on the sequence $(\rho_{k})_{k=1}^{\infty}$.
	\end{itemize}}
	 For the random variable $X_{p}\colon[0,1]\to \R$ defined by
	\begin{equation*}
	X_{p}(t)=\langle{\beta_{p}(t),e_{k_{p}(\gamma(t))}^{\perp}}\rangle \chi_{D_{p}},\qquad t\in[0,1],
	\end{equation*}
	\eqref{eq:expectation} gives $\Exp[\left|X\right|]\leq 4^{-p}$. Moreover, the set $D_{p}$ is contained in $\left\{t\in[0,1] \colon \left|X_{p}(t)\right|>2^{-p}\right\}$; \hl{see \eqref{eq:Dp}}. Applying Markov's Inequality, we conclude
	\begin{equation*}
	\leb(D_{p})\leq\leb(\left\{t\colon \left|X_{p}(t)\right|>2^{-p}\right\})\leq \frac{\Exp[\left|X_{p}\right|]}{2^{-p}}<\frac{4^{-p}}{2^{-p}}=2^{-p}.
	\end{equation*}	
\end{proof}
When studying $\gamma^{-1}(E)$ later on, Lemma~\ref{prop:Dp} will allow us to discard the sets $D_{p}$. In the remaining set we have that $\beta_{p}(t)$ is very close to the direction $e_{k_{p}(\gamma(t))}$ of the $p$-th strip containing $\gamma(t)$. The form of this approximation that we will require is recorded in the following lemma.
\begin{lemma}[Under Hypothesis~\ref{hyp}]\label{lemma:outsideDp}
	Let $t\in\gamma^{-1}(E)\setminus D_{p}$. Then, writing $k_{p}$ for $k_{p}(\gamma(t))$,
	\begin{equation*}
	\abs{\frac{\skp{w,e_{k_{p}}^{\perp}}}{\skp{w,e_{k_{p}}}}-\frac{\skp{\beta_{p}(t),w^{\perp}}}{\skp{\beta_{p}(t),w}}}\leq \frac{2^{-p}}{(1-\eta)(1-\delta)}.
	\end{equation*}
\end{lemma}
\begin{proof}
	We rewrite the considered expression as
	\begin{equation*}
	\left|\frac{\langle{\beta_{p}(t),w}\rangle\langle{e_{k_{p}},w^{\perp}\rangle -\langle{\beta_{p}(t),w^{\perp}}\rangle\langle{e_{k_{p}},w}\rangle}}{\langle{w,e_{k_{p}}}\rangle\langle{\beta_{p}(t),w}\rangle}\right|.
	\end{equation*}
	The numerator above is precisely the determinant of the $2\times 2$ matrix with columns $\beta_{p}(t)$ and $e_{k_{p}}$, which is given in absolute value by $\abs{\langle{\beta_{p}(t),e_{k_{p}}^{\perp}}\rangle}\leq 2^{-p}$. The denominator is bounded below in absolute value by $(1-\eta)(1-\delta)$.
\end{proof}
\subsection{A function with small set of differentiability points inside $E$.}\label{subsec:function}
Our aim is now to construct a Lipschitz function \hl{$h\colon\R^{2}\to\R$} having only a very small set of differentiability points in $E$. The function $h$ will \hl{be defined as the uniform limit of a sequence of functions $h_{n}\colon \R^{2}\to\R$ of the form
\begin{equation*}
h_{n}(z)=\sum_{k=1}^{n}2^{-m_{k-1}(z)}\sigma_{k-1}(z)\varphi_{k}(z),\qquad n\in\N\cup\set{0},
\end{equation*}
}where \hl{$m_{k}\colon \R^{2}\to \N\cup\set{0}$}, \hl{$\sigma_{k}\colon \R^{2}\to \set{-1,1}$}, \hl{$\varphi_{k}\colon \R^{2}\to \R$} are functions to be constructed.

\paragraph{Definition and properties of $\varphi_{k}$.} 
The construction of the functions $\varphi_{k}$ will be intertwined with that of the lines $L_{k}$, widths $\rho_{k}$ and additional sequences of sets $T_{k}\subseteq \R^{2}$ and numbers $\delta_{k}>0$. 
\hl{\begin{itemize}
	\item[(\mylabel{E4}{E4})] Thus, we prescribe here, that the sequences $(L_{k})_{k=1}^{\infty}$ of lines $L_{k}$ and $(\rho_{k})_{k=1}^{\infty}$ of widths $\rho_{k}$ introduced in \eqref{eq:set_E} to define the set $E$, are in fact constructed according to the following procedure. It is a trivial matter to adapt the procedure described below so that the sequences $(L_{k})_{k=1}^{\infty}$ and $(\rho_{k})_{k=1}^{\infty}$ it produces satisfy the existing conditions \eqref{E1}, \eqref{E2} and \eqref{E3}. We spare the details of this.
\end{itemize}}

The construction begins by setting $T_{0}=\emptyset$. Now for $k\geq 1$ and $T_{k-1}$ already defined as a finite union of lines, we choose the line $L_{k}\subseteq \R^{2}$ so that the set $S_{k}:=L_{k}\cap T_{k-1}$ is finite. The number $\delta_{k}>0$ is then chosen small \hl{according to} the cardinality of $S_{k}$ and then $\rho_{k}>0$ is chosen sufficiently small \hl{depending on all previous data}; these conditions will be made precise later \hl{in \eqref{eq:cond_on_deltak} and \eqref{E5}}. We let $\widetilde{\varphi}_{k}\colon \R^{2}\to\R$ be the function uniquely determined by the following conditions:
\begin{enumerate}[(A)]
	\item\label{constant_along_ek} $\widetilde{\varphi}_{k}$ is constant along all lines parallel to \hl{the line $L_{k}$, that is, along all lines parallel to the direction $e_{k}\in S^{1}$.}
	\item\label{along_wperp} Along each line parallel to $w^{\perp}$ the function \hl{$\widetilde{\varphi}_{k}$} is \hl{constantly equal to} $0$ in the lower \hl{connected} component (with respect to the direction $w^{\perp}$) of $\R^{2}\setminus B(L_{k},\rho_{k})$, grows with slope $1$ inside the strip $B(L_{k},\rho_{k})$ and is \hl{constantly equal to} $\frac{2\rho_{k}}{\langle{w,e_{k}}\rangle}$ on the upper \hl{connected} component of $\R^{2}\setminus B(L_{k},\rho_{k})$.
\end{enumerate}
\hl{Note that $\widetilde{\varphi}_{k}$ is affine on each component of $\R^{2}\setminus \partial B(L_{k},\rho_{k})$.} Next, we define \hl{a function} $\varphi_{k}\colon \R^{2}\to \R$ by
\begin{equation*}
\varphi_{k}(z)=\min\left\{\widetilde{\varphi}_{k}(z),2^{-k}\dist(z,T_{k-1})\right\},
\end{equation*}
\hl{where $\dist(z,\emptyset)=\infty$}, and define $T_{k}$ as the minimal \hl{(finite)} union of lines in $\R^{2}$ which contains $T_{k-1}\cup \partial B(L_{k},\rho_{k})$ and for which $\varphi_{k}$ is affine on each \hl{connected} component of $\R^{2}\setminus T_{k}$. This completes the construction.

The next lemma records the important properties of the functions $(\varphi_{k})_{k=1}^{\infty}$:
\begin{lemma}\label{lemma:var}
	For each $k\in\N$ the function $\varphi_{k}\colon\R^{2}\to\R$ has the following properties:
	\begin{enumerate}[(a)]
		\item\label{var:pcwaff} $\varphi_{k}$ is affine on each \hl{connected} component of $\R^{2}\setminus T_{k}$.
		\item\label{var:sup} $\lnorm{\infty}{\varphi_{k}}\leq \lnorm{\infty}{\widetilde{\varphi}_{k}} \leq \frac{2\rho_{k}}{\langle{w,e_{k}}\rangle}\leq \frac{2\rho_{k}}{1-\eta}$.
		\item\label{var:der_strip} For each point $z\in B(L_{k},\rho_{k})\setminus B(S_{k},\delta_{k})$ we have 
		\begin{enumerate}[(i)]
			\item\label{var:der_strip_1} $B(z,\frac{5\rho_{k}}{\sqrt{\eta}})\cap T_{k-1}=\emptyset$,
			\item\label{var:der_strip_2} $\varphi_{k}=\widetilde{\varphi_{k}}$ on $B(z,\frac{5\rho_{k}}{\sqrt{\eta}})$, and
			\item\label{var:der_strip_3} $\displaystyle D\varphi_{k}(z)=w^{\perp}+\frac{\skp{w,e_{k}^{\perp}}}{\skp{w,e_{k}}}\cdot w$.			
		\end{enumerate} 		
		\item\label{var:der_out} For each point $z\in \R^{2}\setminus B(L_{k},\rho_{k})$ at which the derivative of $\varphi_{k}$ exists we have
		\begin{equation*}
		\norm{D\varphi_{k}(z)}\leq 2^{-k}.
		\end{equation*}
		\item\label{var:der_bound} $\lnorm{\infty}{D\varphi_{k}}\leq \sqrt{1+\left(\frac{1}{1-\eta}\right)^{2}}$.
		\item\label{var:nondiff} For each point $z\in B(L_{k},\rho_{k})\setminus B(S_{k},\delta_{k})$ and each direction $v\in S^{1}\setminus \widehat{C}(w,3\sqrt{\eta})$ there exist a point $u\in\R^{2}$, and numbers $\rho_{k}\leq t_{1}\leq t_{2}\leq \frac{2\rho_{k}}{\sqrt{\eta}}$ such that $z\in [u,u+t_{1}v]$ and
		\begin{equation*}
				\left|\frac{\varphi_{k}(u+t_{1}v)-\varphi_{k}(u)}{t_{1}}-\frac{\varphi_{k}(u+t_{2}v)-\varphi_{k}(u)}{t_{2}}\right|\geq \frac{\sqrt{\eta}}{2}.
		\end{equation*}
	\end{enumerate}
\end{lemma}
\begin{proof}
	Properties \eqref{var:pcwaff} and \eqref{var:sup} are immediate from the construction. 
	 For \eqref{var:der_strip} we need to impose a condition on $\rho_{k}$ relative to $\delta_{k}$. Since $T_{k-1}$ is a finite union of lines and $L_{k}\setminus B(S_{k},\delta_{k})$ is a finite union of closed line segments and half-rays not intersecting $T_{k-1}$ the quantity
		\begin{equation}\label{eq:ck}
		c_{k}:=\inf\left\{\dist(x,y)\colon x\in T_{k-1},\,y\in L_{k}\setminus B(S_{k},\delta_{k}\hl{/2})\right\}
		\end{equation}
		is positive. \hl{Referring to the paragraph following \eqref{E4}, we also note that \hl{$S_{k}$, $\delta_{k}$ and} $c_{k}$ \hl{are} determined before $\rho_{k}$ is chosen. Therefore, we may impose conditions on $\rho_{k}$ according to $S_{k}$, $\delta_{k}$ and $c_{k}$, as we do in the next passage of text. All of these imposed conditions will then be collated in \eqref{E5} below.} 
		
		For all $z\in B(L_{k},\rho_{k})\setminus B(S_{k},\delta_{k})$ we have that \hl{$\proj_{L_{k}}(z)\in L_{k}\setminus B(S_{k},\delta_{k}/2)$, when we impose the condition $\rho_{k}<\delta_{k}/2$. Therefore,}
		\begin{equation*}
		\dist(z,T_{k-1})\geq c_{k}-\rho_{k}=\left(\hl{\frac{c_{k}}{\rho_{k}}-1}\right)\rho_{k}>\frac{5\rho_{k}}{\sqrt{\eta}},
		\end{equation*}
		where the last inequality is another condition on \hl{$\rho_{k}$.} This proves \eqref{var:der_strip_1}. Given $z\in B(L_{k},\rho_{k})\setminus B(S_{k},\delta_{k})$ and $z'\in B(z,\frac{5\rho_{k}}{\sqrt{\eta}})$ we have
	\begin{equation*}
	\dist(z',T_{k-1})\geq \dist(z,T_{k-1})-\frac{5\rho_{k}}{\sqrt{\eta}}\geq \left(\hl{\frac{c_{k}}{\rho_{k}}-1}-\frac{5}{\sqrt{\eta}}\right)\rho_{k}>\frac{2^{k+1}\rho_{k}}{1-\eta}\geq 2^{k}\lnorm{\infty}{\widetilde{\varphi}_{k}}
	\end{equation*}
	where the penultimate inequality is a further condition on \hl{$\rho_{k}$}. We deduce that $\widetilde{\varphi}_{k}(z')<2^{-k}\dist(z',T_{k-1})$. Hence, $\varphi_{k}(z')=\widetilde{\varphi}_{k}(z')$. This proves \eqref{var:der_strip_2}, after which \eqref{var:der_strip_3} derives easily from the defining properties \eqref{constant_along_ek} and \eqref{along_wperp} of $\widetilde{\varphi}_{k}$.  
	\hl{\begin{itemize}
			\item[(\mylabel{E5}{E5})] To summarise, the proof of \eqref{var:der_strip_1}--\eqref{var:der_strip_3} given above requires the additional condition
			\begin{equation*}
			\rho_{k}<\min\set{\frac{\delta_{k}}{2},\frac{c_{k}}{\frac{2^{k+1}}{1-\eta}+1+\frac{5}{\sqrt{\eta}}}}.
			\end{equation*}
			on $\rho_{k}$, where $c_{k}$ is defined in \eqref{eq:ck}.
	\end{itemize}}

	For \eqref{var:der_out} and \eqref{var:der_bound} we observe that the plane $\R^{2}$ may be decomposed as a union of finitely many (possibly unbounded) polygons, that is finite intersections of half-spaces, on each of which $\varphi_{k}$ is affine and either $\varphi_{k}=\widetilde{\varphi_{k}}$ or $\varphi_{k}=2^{-k}\dist(\cdot,T_{k-1})$. The inequalities of \eqref{var:der_out} and \eqref{var:der_bound} are readily verified for both cases.
	
	Finally we verify \eqref{var:nondiff}: Given $z\in B(L_{k},\rho_{k})\setminus B(S_{k},\delta_{k})$ and $v\in S^{1}\setminus \widehat{C}(w,3\sqrt{\eta})$ we choose $u\in L_{k}$ and so that $z\in u+\R v$. We assume, without loss of generality that $z\in u+[0,\infty)v$ and let $t_{1}$ and $t_{2}$ be defined by the conditions
	\begin{equation*}
	u+t_{j}v\in \partial B(L_{k},j\rho_{k}),\qquad j=1,2.
	\end{equation*}
	Clearly $t_{1}\geq \rho_{k}$, $t_{2}=2t_{1}$ and $z\in [u,u+t_{1}v]$. From elementary geometric considerations and the conditions $e_{k}\in C(w,\eta)$ and $v\in S^{d-1}\setminus \widehat{C}(w,3\sqrt{\eta})$ we derive 
	\begin{equation*}
	\left|\langle{v,e_{k}^{\perp}}\rangle\right|=\frac{\rho_{k}}{t_{1}},\qquad \left|\langle{v,e_{k}}\rangle\right|<(1-3\sqrt{\eta})+\sqrt{\eta(2-\eta)}<1-\sqrt{\eta}.
	\end{equation*}
	Together with the identity $\left|\langle{v,e_{k}^{\perp}}\rangle\right|^{2}+\left|\langle{v,e_{k}}\rangle\right|^{2}=1$, this leads to 
	\begin{equation*}
	t_{1}<\frac{\rho_{k}}{(2\sqrt{\eta}-\eta)^{1/2}}\leq \frac{\rho_{k}}{\sqrt{\eta}}.
	\end{equation*}
	Now, from the definition of $\widetilde{\varphi}_{k}$ it is clear that 
	\begin{equation*}
	\left|\widetilde{\varphi}_{k}(u+t_{1}v)-\widetilde{\varphi}_{k}(u)\right|\geq \rho_{k},\quad \text{and }\quad \widetilde{\varphi}_{k}(u+2t_{1}v)=\widetilde{\varphi}_{k}(u+t_{1}v).
	\end{equation*}
	Moreover, we note that $[u,u+2t_{1}v]\subseteq B(z,\frac{5\rho_{k}}{\sqrt{\eta}})$. Therefore, using \eqref{var:der_strip_2} we have that $\varphi_{k}=\widetilde{\varphi}_{k}$ on $[u,u+2t_{1}v]$. We deduce
	\begin{multline*}
	\left|\frac{\varphi_{k}(u+t_{1}v)-\varphi_{k}(u)}{t_{1}}-\frac{\varphi_{k}(u+2t_{1}v)-\varphi_{k}(u)}{2t_{1}}\right|\\
	=\left|\frac{\widetilde{\varphi}_{k}(u+t_{1}v)-\widetilde{\varphi}_{k}(u)}{t_{1}}-\frac{\widetilde{\varphi}_{k}(u+2t_{1}v)-\widetilde{\varphi}_{k}(u)}{2t_{1}}\right|\geq\frac{\rho_{k}}{2t_{1}}\geq \frac{\sqrt{\eta}}{2}.
	\end{multline*}
\end{proof}
\paragraph{Definition and properties of \hl{$\sigma_{k}\colon\R^{2}\to\set{-1,1}$}.} For each $k\in\N$ we define \hl{the function} $\sigma_{k}\colon \R^{2}\to\hl{\set{-1,1}}$ by
\begin{equation*}
\sigma_{k}(z)=(-1)^{p}
\end{equation*}
where $p\in\N$ is the unique integer satisfying $k_{p-1}(z)\leq k <k_{p}(z)$.

\begin{lemma}\label{lemma:sigma_k}
	For each $k\in\N$, $\sigma_{k}$ is constant on each \hl{connected} component of the set $\R^{2}\setminus \bigcup_{j=1}^{k}\partial B(L_{j},\rho_{j})$.
\end{lemma}
\begin{proof}
	It is clear that $\sigma_{0}\equiv -1$. Let $k\geq 1$ and suppose that $\sigma_{k-1}$ is constant on each \hl{connected} component of $\R^{2}\setminus \bigcup_{j=1}^{k-1}\partial B(L_{j},\rho_{j})$. Given $z\in\R^{2}$, let $p\in \N$ be the unique integer with
	\begin{equation*}
	k_{p-1}(z)\leq k-1<k_{p}(z),
	\end{equation*}
	determining that $\sigma_{k-1}(z)=(-1)^{p}$. The inequalities above express that the point $z$ belongs to precisely $p-1$ strips $B(L_{j},\rho_{j})$ with index $j\in\left\{1,\ldots,k-1\right\}$. Hence, $k_{p}(z)=k$ if $z\in B(L_{k},\rho_{k})$ and $k_{p}(z)>k$ otherwise. From this consideration it follows that 
	\begin{equation*}
	\sigma_{k}(z)=\begin{cases}
	(-1)^{p+1} & \text{if }z\in B(L_{k},\rho_{k}),\\
	(-1)^{p}=\sigma_{k-1}(z) & \text{otherwise.}
	\end{cases}
	\end{equation*}
	This completes the induction step, proving the lemma.
\end{proof}
\paragraph{Definition and properties of \hl{$m_{k}\colon\R^{2}\to\N\cup\set{0}$} and \hl{$h_{k}\colon \R^{2}\to \R$}.}
The functions \hl{$m_{k}\colon\R^{2}\to\N\cup\set{0}$} and \hl{$h_{k}\colon\R^{2}\to \R$} are defined for each $k\in\N$ inductively as follows. Set $m_{0}=h_{0}=0$ on the whole of $\R^{2}$. If $k\geq 1$ and the functions $m_{k-1}$ and $h_{k-1}$ are already defined, we let 
\begin{equation*}
h_{k}(z)=h_{k-1}(z)+2^{-m_{k-1}(z)}\sigma_{k-1}(z)\varphi_{k}(z),\qquad z\in \R^{2}.
\end{equation*}
Finally, whenever $h_{0},\ldots,h_{k}$ and $m_{0},\ldots,m_{k-1}$ are already defined we let
\begin{equation}\label{eq:def_jk}
j_{k}(z):=\max\left\{\hl{j\in\set{1,2,\ldots,k-1}}\colon m_{j}(z)\neq m_{j-1}(z)\right\},
\end{equation}
where we interpret the maximum as zero if the set considered is empty. For $z\in\R^{2}$ let
\begin{equation}\label{def:m_k}
m_{k}(z)=\begin{cases}
m_{k-1}(z)+1 & \text{if }z\in \R^{2}\setminus T_{k}\text{ and }\norm{Dh_{k}(z)-Dh_{j_{k}(z)}(z)}>\varepsilon(m_{j_{k}(z)}(z)),\\
m_{k-1}(z) & \text{ otherwise,}
\end{cases}
\end{equation}
where $(\varepsilon(n))_{n=0}^{\infty}$ is a sequence of positive real numbers, \hl{which will be subject to precisely two simple, additional conditions \eqref{eps1} and \eqref{eps2} imposed at the moments when they are required later on.}

% the necessary conditions on which we specify at the relevant point later on.

We summarise the important properties of the functions $h_{k}$ and $m_{k}$:
\begin{lemma}\label{lemma:h_k}
	\begin{enumerate}[(a)]
		\item\label{h_m_piecewise} For each $k$ and on each \hl{connected} component of $\R^{2}\setminus T_{k}$ we have that \hl{$h_{k}\colon\R^{2}\to\R$} is affine and \hl{$m_{k}\colon\R^{2}\to\N\cup\set{0}$} is constant.
		\item\label{m_lsc}For all $k$ the function \hl{$m_{k}\colon \R^{2}\to\N\cup\set{0}$} is lower semi-continuous.
		\item\label{Dhl-Dhk} For all $l\geq k$ and all $z\in\R^{2}\setminus T_{l}$ we have
		\begin{equation*}
		\norm{Dh_{l}(z)-Dh_{k}(z)}\leq K(\eta)\sum_{j=m_{k}(z)}^{\infty}\left(2^{-j}+\varepsilon(j)\right),
		\end{equation*}
		where $K(\eta)$ denotes a constant depending only on $\eta$.
		\item\label{Dhl-Dhk_2}
		For all $l\geq k$ and all $z\in\R^{2}\setminus T_{l}$ we have
		\begin{equation*} \norm{Dh_{l}(z)-Dh_{k}(z)}\leq 2^{-m_{k}}+\norm{\sum_{\left\{s\colon k<k_{s}\leq l\right\}}2^{-m_{k_{s}-1}}\sigma_{k_{s}-1}D\varphi_{k_{s}}}.
		\end{equation*}
	\end{enumerate}
\end{lemma}
\begin{proof}
	The statement \eqref{h_m_piecewise} is trivially valid for $m_{0}\equiv h_{0}\equiv 0$. Assume now that \eqref{h_m_piecewise} holds for the objects $T_{j}$, $m_{j}$ and $h_{j}$ for all $j<k$. Then by Lemma~\ref{lemma:sigma_k} and the construction of the sets $T_{j}$ we deduce that $h_{k}$ is affine on each \hl{connected} component of $\R^{2}\setminus T_{k}$. In other words, $Dh_{k}$ is constant on each \hl{connected} component of $\R^{2}\setminus T_{k}$. Moreover, we observe that the function $j_{k}$ is constant on each \hl{connected} component of $\R^{2}\setminus T_{k}$. Referring to the definition of $m_{k}$ above, we conclude that the set of points where $m_{k}\neq m_{k-1}$ (meaning $m_{k}=m_{k-1}+1$), is a union of \hl{connected} components of $\R^{2}\setminus T_{k}$. Applying the induction hypothesis, the proof of \eqref{h_m_piecewise} complete. A simple induction argument based on \eqref{def:m_k} also verifies \eqref{m_lsc}.
	
	We turn our attention to \eqref{Dhl-Dhk}. Let $l\geq k$ and $z\in\R^{2}\setminus T_{l}$. Then both derivatives $Dh_{l}(z)$ and $Dh_{k}(z)$ exist. In what follows we use the fact that all functions $\sigma_{t}$, $m_{t}$, $j_{t}$ and $Dh_{t}$ with index $t\leq l$ are constant on the \hl{connected} component of $\R^{2}\setminus T_{l}$ containing $z$. Since we are only concerned with a neighbourhood of $z$, we will sometimes omit the argument of such functions. We also allow the constant $K(\eta)$ to change in each occurence. Let $(r_{n})_{n\geq 0}$ be the finite sequence of minimal indices $r_{n}$ satisfying $m_{r_{n}}(z)=n$ and $r_{n}\leq l$. By the definition~\hl{\eqref{eq:def_jk}} of the functions $j_{\hl{i}}(z)$ we have
	\begin{equation*}
	j_{\hl{i}}(z)=r_{n-1}\qquad r_{n-1}<\hl{i}\leq r_{n},\qquad n\geq 1.
	\end{equation*}
	Now, combining Lemma~\ref{lemma:var}, \eqref{var:der_bound} and the rule \eqref{def:m_k} governing the growth of the sequence $(m_{k}(z))_{k=0}^{\infty}$, we deduce 
	\begin{equation*}
	\norm{Dh_{r_{n}}-Dh_{r_{n-1}}}\leq \norm{2^{-(n-1)}\sigma_{r_{n}-1}D\varphi_{r_{n}}}+\norm{Dh_{r_{n}-1}-Dh_{r_{n-1}}}\leq
	 K(\eta)2^{-(n-1)}+\varepsilon(n-1),
	\end{equation*}
	for all $n\geq 1$. Choose $s$ and $t$ maximal with $r_{s}\leq l$ and $r_{t}\leq k$. Then $m_{l}(z)=m_{r_{s}}(z)=s$ and $m_{k}(z)=m_{r_{t}}(z)=t$. Moreover, from \eqref{def:m_k} and the bound above we get
	\begin{multline*}
	\norm{Dh_{l}(z)-Dh_{k}(z)}\leq \norm{Dh_{l}-Dh_{r_{s}}}+\norm{Dh_{r_{s}}-Dh_{r_{t}}}+\norm{Dh_{r_{t}}-Dh_{k}}\\
	\leq \varepsilon(s)+K(\eta)\sum_{j=t}^{s-1}(2^{-j}+\varepsilon(j))+\varepsilon(t)\leq K(\eta)\sum_{j=m_{k}(z)}^{\infty}2^{-j}+\varepsilon(j).
	\end{multline*}
	This proves \eqref{Dhl-Dhk}. For \eqref{Dhl-Dhk_2} we note that 
	\begin{equation*}
	\left\|Dh_{l}(z)-Dh_{k}(z)\right\|=\left\|\sum_{j=k+1}^{l}2^{-m_{j-1}}\sigma_{j-1}D\varphi_{j}(z)\right\|.
	\end{equation*}
	The above sum can be split into two parts: firstly the sum over those indices $j$ for which $z\notin B(L_{j},\rho_{j})$ and secondly, the sum over those indices $j$ of the form $k_{s}(z)$. The inequality of \eqref{Dhl-Dhk_2} is obtained simply by leaving the second sum unchanged and bounding the first sum by $2^{-m_{k}}$ using Lemma~\ref{lemma:var} \eqref{var:der_out}.	
\end{proof}

\paragraph{Properties of $h$.}
We bring together all the \hl{pieces} and derive the important properties of the function \hl{$h\colon\R^{2}\to\R$ given by}
\begin{equation}\label{eq:function_h}
h(z)=\sum_{k=1}^{\infty}2^{-m_{k-1}(z)}\sigma_{k-1}(z)\varphi_{k}(z)=\lim_{k\to \infty}h_{k}(z),\qquad z\in \R^{2}.
\end{equation}
\begin{lemma}\label{lemma:h_lip}
The function \hl{$h\colon\R^{2}\to\R$} is well-defined and Lipschitz. 
%Moreover
%	\begin{equation*}
%	\left\|Dh_{k}(z)\right\|\leq C_{m_{k}}
%	\end{equation*}
%	wherever this derivative exists.
\end{lemma}
\begin{proof}
	We will first verify that $h$ is well-defined and continuous. Lemmas~\ref{lemma:sigma_k} and~\ref{lemma:h_k}~\eqref{h_m_piecewise}, the continuity of $\varphi_{k}$ and the fact that $\varphi_{k}=0$ on each line in the family $T_{k-1}$ ensure that each summand $2^{-m_{k-1}}\sigma_{k-1}\varphi_{k}$ is continuous. Using $\left\|\varphi_{k}\right\|_{\infty}\leq \frac{2\rho_{k}}{1-\eta}$ (Lemma~\ref{lemma:var}~\eqref{var:sup}), the sequence of partial sums $h_{k}$ is easily seen to converge uniformly to $h$ and so $h$ is well-defined and continuous as well. To show that $h$ is Lipschitz, it suffices to show that the functions $h_{k}$ are Lipschitz with uniformly bounded Lipschitz constants. Since the functions $h_{k}$ are continuous and piecewise affine, it suffices to verify that their derivatives $Dh_{k}$ are uniformly bounded. 
	\hl{\begin{itemize}
		\item[(\mylabel{eps1}{$\varepsilon$1})] This is implied by Lemma~\ref{lemma:h_k}~\eqref{Dhl-Dhk}, when we prescribe that the sequence $(\varepsilon(n))_{n=0}^{\infty}$ is summable.
	\end{itemize}}
	
\end{proof}
\paragraph{Sets $G$, $H$ and $F_{m}$.} We now introduce two sets $G,H$ which will be shown to cover the set of points inside $E$ where $h$ has a directional derivative in any direction outside of a small double sided cone. We let
\begin{align}\label{eq:sets_G_H}
G:=\bigcap_{n=1}^{\infty}\bigcup_{k=n}^{\infty}B(S_{k},\delta_{k}),\qquad H:=\left\{z\in E\colon \lim_{k\to \infty}m_{k}(z)=\infty\right\}.
\end{align}
Note that the complement of $G\cup H$ inside of $E$ may be covered by the sets
\begin{equation*}
F_{m}:=\left\{z\in \R^{2}\colon \lim_{k\to \infty}m_{k}(z)\leq m\right\}\setminus G,\qquad m\in\N.
\end{equation*}
The topological properties of the sets $G$, $H$ and $F_{m}$ will be important later on. We note that $G$ and $H$ are both $G_{\delta}$ sets. For $G$ this is clear; for $H$ it follows easily from the fact that $E$ is $G_{\delta}$ and Lemma~\ref{lemma:h_k}, \eqref{m_lsc}. Using Lemma~\ref{lemma:h_k}, \eqref{m_lsc} again, we deduce that each set $F_{m}$ is $F_{\sigma}$.

In the next lemma we show that $h$ is nowhere differentiable inside each set $E\cap F_{m}$. Moreover, we obtain a uniform bound on the degree of non-differentiability.
\begin{lemma}\label{lemma:cover_diff}
	Let $m\in\N$, $z\in E\cap F_{m}$ and $v\in S^{1}\setminus \widehat{C}(w,3\sqrt{\eta})$. Then
	\begin{equation*}
	\limsup_{\varepsilon\to 0}\zeta(h,z,\varepsilon,v)\geq \frac{2^{-m}\sqrt{\eta}}{4}.
	\end{equation*}		
	Hence, in the set $E\setminus (G\cup H)=E\cap \bigcup_{m=1}^{\infty}F_{m}$ we have that $h$ is nowhere differentiable and has no directional derivatives in any direction \hl{outside of} $\widehat{C}(w,3\sqrt{\eta})$.
\end{lemma}
\begin{proof}
Fixing $\varepsilon>0$, we need to find two line segments passing through $z$, parallel to $v$ and of length at most $\varepsilon$ on which $h$ has slopes differing by at least $2^{-m}\sqrt{\eta}/4$. Since $z\in E$, the numbers $k_{p}:=k_{p}(z)$ are finite and $z\in B(L_{k_{p}},\rho_{k_{p}})$ for all $p\in\N$. Since $z\notin G$, we may choose $p$ sufficiently large so that for $k:=k_{p}$ we have $z\notin B(S_{k},\delta_{k})$. We additionally choose $p$ sufficiently large so that $\frac{2\rho_{k}}{\sqrt{\eta}}<\varepsilon$ and 
\begin{equation*}
m_{k-1}(z)=\max_{j\in\N}m_{j}=:\widetilde{m}\leq m.
\end{equation*}
By Lemma~\ref{lemma:var}, \eqref{var:der_strip_1}, Lemma~\ref{lemma:sigma_k} and Lemma~\ref{lemma:h_k}, \eqref{h_m_piecewise} the functions $\sigma_{k-1}$ and $m_{k-1}$ are constant on $B\left(z,\frac{5\rho_{k}}{\sqrt{\eta}}\right)$. Therefore, for all $y\in B\left(z,\frac{5\rho_{k}}{\sqrt{\eta}}\right)$ we have
	\begin{equation}\label{eq:h_decomp}
	h(y)=h_{k-1}(y)+2^{-\widetilde{m}}\sigma_{k-1}(z)\varphi_{k}(y)+\underbrace{\sum_{j=k+1}^{\infty}2^{-\hl{m_{j-1}(y)}}\sigma_{j-1}(y)\varphi_{j}(y)}_{=(h-h_{k})(y)}.
	\end{equation}
	Moreover, by Lemma~\ref{lemma:h_k}, \eqref{h_m_piecewise}, the function $h_{k-1}$ is affine on $B\left(z,\frac{5\rho_{k}}{\sqrt{\eta}}\right)$.
	
	Let $u,t_{1},t_{2}$ be given by the conclusion of Lemma~\ref{lemma:var}~\eqref{var:nondiff} for $\varphi_{k}$, $z$ and $v$. Then the segments $[u,u+t_{1}v]$ and $[u,u+t_{2}v]$ both contain $z$, have length at most $t_{2}\leq \frac{2\rho_{k}}{\sqrt{\eta}}<\varepsilon$ and are therefore contained in $B\left(z,\frac{5\rho_{k}}{\sqrt{\eta}}\right)$. Hence $h_{k-1}$ restricted to $[u,u+t_{2}v]$ is affine and 
	\begin{equation*}
	\left|\frac{h_{k-1}(u+t_{1}v)-h_{k-1}(u)}{t_{1}}-\frac{h_{k-1}(u+t_{2}v)-h_{k-1}(u)}{t_{2}}\right|=0.
	\end{equation*}
	The corresponding difference \hl{of slopes} for the tail sum in \eqref{eq:h_decomp} \hl{may} be bounded above using $\left|\sigma_{j}\right|\equiv 1$, $\lnorm{\infty}{\varphi_{j}}\leq \frac{2\rho_{j}}{1-\eta}$ and $t_{1},t_{2}\geq \rho_{k}$, leading to
	\begin{multline*}
	\left|\frac{(h-h_{k})(u+t_{1}v)-(h-h_{k})(u)}{t_{1}}-\frac{(h-h_{k})(u+t_{2}v)-(h-h_{k})(u)}{t_{2}}\right|\\
	\leq \frac{4}{\rho_{k}}\sum_{j=k+1}^{\infty}2^{-\widetilde{m}}\cdot \frac{2\rho_{j}}{1-\eta}\leq \frac{2^{-\widetilde{m}}\sqrt{\eta}}{16},
	\end{multline*}
	where the last inequality imposes a condition on the sequence $(\rho_{j})_{j=1}^{\infty}$. 
	\hl{\begin{itemize}
		\item[(\mylabel{E6}{E6})] This condition may be written equivalently as
		\begin{equation*}
		\frac{1}{\rho_{k}}\cdot\sum_{j=k+1}^{\infty}\rho_{j}\leq\frac{\sqrt{\eta}(1-\eta)}{64}.
		\end{equation*}
	\end{itemize}}
	Now combining the two \hl{difference of slopes} bounds above with that of Lemma~\ref{lemma:var}~\eqref{var:nondiff} we obtain
	\begin{equation*}
	\left|\frac{h(u+t_{1}v)-h(u)}{t_{1}}-\frac{h(u+t_{2}v)-h(u)}{t_{2}}\right|\geq \frac{2^{-\widetilde{m}}\sqrt{\eta}}{2}-\frac{2^{-\widetilde{m}}\sqrt{\eta}}{16}\geq \frac{2^{-m}\sqrt{\eta}}{4},
	\end{equation*}
	which completes the proof.
	%We will demonstrate the non-differentiability of $h$ in direction $v$ at $z$ by finding $u\in\R^{2}$ and $0<t_{1}\leq t_{2}\leq \varepsilon$ so that the segment $[u,u+t_{2}v]$ contains $z$ and has length at most $\varepsilon$, but sees two very different slopes of the function $h$. More precisely we will show that the quantity
	%	\begin{equation*}
	%	\left|\frac{h(u+t_{1}v)-h(u)}{t_{1}}-\frac{h(u+t_{2}v)-h(u)}{t_{2}}\right|
	%	\end{equation*}
	%	is bounded below by some constant independent of $\varepsilon$.
%	For \eqref{der_ex} we observe that for each $k\in\N$ the quantity
%	\begin{equation*}
%	\left|\frac{g(z+te_{k})+g(z-te_{k})-2g(z)}{t}\right|
%	\end{equation*}
%	is zero for both $g=h_{k-1}$ and $g=2^{-m}\sigma_{k-1}\varphi_{k}$ and $t\in(0,5\rho_{k})$. Therefore, for all $t\in(\rho_{k},5\rho_{k})$ we have
%	\begin{multline*}
%	\left|\frac{h(z+te_{k})+h(z-te_{k})-h(z)}{t}\right|\\
%	\leq 2^{-m}\sum_{j=k+1}^{\infty}\left|\frac{\varphi_{j}(z+te_{k})+\varphi_{j}(z-te_{k})-\varphi_{j}(z)}{t}\right|\leq 2^{-m}\cdot \frac{4\sum_{j=k+1}^{\infty}\lnorm{\infty}{\varphi_{j}}}{\rho_{k}}\\
%	\leq \frac{4}{2^{m}\rho_{k}}\sum_{j=k+1}^{\infty}\rho_{j}=:\varepsilon_{k}\to 0.
%	\end{multline*}
	\end{proof}
We now prove that $h$ is differentiable everywhere in the set $H\setminus G$.
\begin{lemma}\label{lemma:h_diff}
	Let $z\in H\setminus G$. Then $h$ is differentiable at $z$.
\end{lemma}
\begin{proof}
%	The set $H$ is a subset of $E$. Therefore, for each $p\in\N$, we have $k_{p}:=k_{p}(z)<\infty$ with $z\in B(L_{k_{p}},\rho_{k_{p}})$. Since $z\notin G$ there exists $n\in\N$ such that $z\notin \bigcup_{k=n}^{\infty}B(S_{k},\delta_{k})$. Hence, for all sufficiently large $p\in\N$ we have $z\in B(L_{k_{p}},\rho_{k_{p}})\setminus B(S_{k_{p}},\delta_{k_{p}})$. 
	Since $z\in H\setminus G\subseteq E\setminus G$ we have $k_{p}:=k_{p}(z)<\infty$ for all $p\in\N$ and that $z\in B(L_{k_{p}},\rho_{k_{p}})\setminus B(S_{k_{p}},\delta_{k_{p}})$ for all sufficiently large $p\in\N$. By Lemma~\ref{lemma:var}~\eqref{var:der_strip_1} and Lemma~\ref{lemma:h_k}~\eqref{h_m_piecewise} there is, for each such $p$, a neighbourhood $B_{p}:=B(z,\frac{5\rho_{k_{p}}}{\sqrt{\eta}})$ of $z$ on which the function $h_{k_{p}-1}$ is affine. In particular each function $h_{k_{p}-1}$ is differentiable at $z$. Set $g_{p}=h_{k_{p}-1}$. 
	
	By Lemma~\ref{lemma:h_k}, \eqref{h_m_piecewise} we have that $m_{k_{p}-1}$ is constant on the set $B_{p}$. Hence, from the inequality of Lemma~\ref{lemma:h_k}~\eqref{Dhl-Dhk}, we may derive
	\begin{equation*}
	\lip((g_{q}-g_{p})|_{B_{p}})\leq K(\eta)\sum_{j=m_{k_{p}-1}(z)}^{\infty}2^{-j}+\varepsilon(j)
	\end{equation*}
	for $q\geq p$. Since this bound is independent of $q\geq p$ and the functions $g_{q}$ converge uniformly to $h$ as $q\to\infty$, we obtain
	\begin{equation*}
	\lip((h-g_{p})|_{B_{p}})\leq K(\eta)\sum_{j=m_{k_{p}-1}(z)}^{\infty}2^{-j}+\varepsilon(j).
	\end{equation*}
	As $p\to\infty$ the lower index $m_{k_{p}-1}(z)$ in the sums above tends to $\infty$, because $z\in H$. We conclude that 
	\begin{equation*}
	\lim_{p\to\infty}\sup_{q\geq p}\lip((g_{q}-g_{p})|_{B_{p}})=\lim_{p\to\infty}\lip((h-g_{p})|_{B_{p}})=0.
	\end{equation*}
	Moreover, for any $q\geq p$ we have $\norm{Dg_{q}(z)-Dg_{p}(z)}\leq \lip((g_{q}-g_{p})|_{B_{p}})$. Therefore, the sequence $(Dg_{p}(z))$ is a Cauchy sequence. Let $L\in\R^{2}$ denote its limit.
	
	We are now ready to verify the differentiability of $h$ at $z$ with $Dh(z)=L$. Let $e\in S^{1}$ and $\varepsilon>0$. Now choose $p$ large enough so that $\lip((h-g_{p})|_{B_{p}})<\varepsilon/3$ and $\norm{Dg_{p}(z)-L}\leq \varepsilon/3$. Choose $\delta_{p}>0$ small enough so that $B(z,\delta_{p})\subseteq B_{p}$. In particular, this ensures that $g_{p}$ is affine on the ball of radius $\delta_{p}$ around $z$. Now, for all $t\in(-\delta_{p},\delta_{p})$ we have
	\begin{multline*}
	\abs{h(z+te)-h(z)-tL(e)}\leq \abs{(h-g_{p})(z+te)-(h-g_{p})(z)}\\+\abs{g_{p}(z+te)-g_{p}(z)-tDg_{p}(z)(e)}+\abs{tDg_{p}(z)(e)-tL(e)}\\
	\leq \frac{\varepsilon}{3}\abs{t}+0+\frac{\varepsilon}{3}\abs{t}<\varepsilon \abs{t}.
	\end{multline*}
	
\end{proof}

\subsection{Pure Unrectifiability}
To complete the proof of Theorem~\ref{thm:main_result}, we show that the set $G\cup H$ is purely unrectifiable.
\begin{lemma}\label{lemma:G_zero}
	Let $\gamma\colon I\to\R^{2}$ be a $\C^{1}$ curve, $v\in S^{\hl{1}}$ and $\theta\in(0,1)$ such that
	\begin{equation*}
	\gamma'(t)\in C(v,\theta) \qquad \text{for all $t\in I$.}
	\end{equation*}
	Then $\leb\left(\gamma^{-1}(G)\right)=0$. Moreover, for any one dimensional subspace $U\subseteq \R^{2}$ the projection $\pi_{U}(G)$ has $1$-dimensional Lebesgue measure zero.
\end{lemma}
\begin{proof}
	Fix $\varepsilon>0$. Imposing the condition
	\begin{equation}\label{eq:cond_on_deltak}
	\sum_{k=1}^{\infty}\left|S_{k}\right|\delta_{k}<\infty
	\end{equation}
	on the sequence $(\delta_{k})_{k=1}^{\infty}$ \hl{(as we may according to the construction of \eqref{E4})}, we can choose $n\in\N$ sufficiently large so that 
	\begin{equation*}
	\sum_{k=n}^{\infty}\left|S_{k}\right|\delta_{k}<\frac{(1-\theta)\varepsilon}{2}.
	\end{equation*}
	Then for each $k\geq n$ and each point $x\in S_{k}$, we may apply Lemma~\ref{lemma:crossing} with $W=B(x,\delta_{k})$ to get that $\leb(\gamma^{-1}(B(x,\delta_{k})))\leq \frac{2\delta_{k}}{1-\theta}$. Summing this inequality over all $x\in S_{k}$ and then all $k\geq n$ gives
	\begin{equation*}
	\gamma^{-1}\left(\bigcup_{k=n}^{\infty}B(S_{k},\delta_{k})\right)\leq \frac{2}{1-\theta}\cdot \sum_{k=n}^{\infty}\left|S_{k}\right|\delta_{k}<\varepsilon.
	\end{equation*}
	For the `moreover' part, it suffices to observe that for any $n\geq 1$ the sum $\sum_{k=n}^{\infty}\abs{S_{k}}\cdot2\delta_{k}$ is an upper bound on the one-dimensional Lebesgue measure of any projection $\pi_{U}(G)$.
\end{proof}
Lemma~\ref{lemma:G_zero} clearly implies that the set $G$ is purely unrectifiable. Thus, we are left needing to prove the pure unrectifiability of $H\setminus G$. For a given $\C^{1}$ curve $\gamma\colon I\to \R^{2}$ with some mild restrictions we will show that the set of points $t\in I$ for which $\gamma(t)\in H\setminus G$ may be modelled by the set of points at which some martingale (see Definition~\ref{def:martingale}) associated to $\gamma$ becomes large. We then appeal to martingale theory to argue that such a set is small in measure. The quantity considered in the next lemma for points $z=\gamma(t)\in E$ will be well approximated, as a consequence of Lemma~\ref{lemma:outsideDp}, by the aforementioned martingale.
\begin{lemma}\label{lemma:der_grow}
	Let $z\in H\setminus G$. Then, writing $k_{s}:=k_{s}(z)$, we have
	\begin{equation*}
	\sup_{q\in\N}\abs{\sum_{s=0}^{2q-1}(-1)^{s}\frac{\skp{w,e_{k_{s}}^{\perp}}}{\skp{w,e_{k_{s}}}}}=\infty.
	\end{equation*}
\end{lemma}
Let us explain informally the idea behind the present lemma. Since $z\notin G$ all derivatives $D\varphi_{k_{r}}$ with $r$ sufficiently large have the form of Lemma~\ref{lemma:var}~\eqref{var:der_strip_3}. Hence they all have component $1$ in the $w^{\perp}$ direction. In the summands $2^{-m_{k_{s}}-1}\sigma_{k_{s}-1}\varphi_{k_{s}}$ of $h$ \hl{(see \eqref{eq:function_h})}, the alternating factor $\sigma_{k_{s}-1}=(-1)^{s}$ ensures that the sum of these derivative components in the $w^{\perp}$ direction is alternating and therefore cannot get large. On the other hand, $z$ being in $H$ requires that $m_{k}(z)$ grows to infinity \hl{(see \eqref{eq:sets_G_H})}. The growth of $m_{k}(z)$ is induced by growth of \hl{the} derivative of the partial sums $h_{k}$ \hl{(see \eqref{def:m_k})}. With the derivative of these sums in the $w^{\perp}$ direction staying small, we conclude that their derivative in the $w$ direction must become large and so we derive a lower bound on the sum of the derivative components in the $w$ direction, i.e. the quantity $\left|\sum(-1)^{s}\frac{\langle{w,e_{k_{s}}^{\perp}}\rangle}{\langle{w,e_{k_{s}}}\rangle}\right|$. 

We now present this argument formally.
\begin{proof}[Proof of Lemma~\ref{lemma:der_grow}]
	It is sufficient to prove 
	\begin{equation}\label{eq:sup}
	\sup_{p<q}\abs{\sum_{p< s\leq q}(-1)^{s}\frac{\skp{w,e_{k_{s}}^{\perp}}}{\skp{w,e_{k_{s}}}}}=\infty.
	\end{equation}
	Let $(r_{n})_{n=1}^{\infty}$ be the sequence of minimal indices $r_{n}=r_{n}(z)$ with $m_{r_{n}}(z)=n$. The rule \eqref{def:m_k} governing the growth of $m_{k}(z)$ implies that all derivatives $Dh_{r_{n}}(z)$ exist and 
	\begin{equation*}
	\norm{Dh_{r_{n+1}}(z)-Dh_{r_{n}}(z)}>\varepsilon(m_{r_{n}}(z))=\varepsilon(n)
	\end{equation*}
	for all $n$. Since $z\in E\setminus G$ we have that $z\in B(L_{k_{s}},\rho_{k_{s}})\setminus B(S_{k_{s}},\rho_{k_{s}})$ for all sufficiently large $k_{s}$. Hence for all sufficiently large $k_{s}$ we have an expression for the derivative $D\varphi_{k_{s}}(z)$ given by Lemma~\ref{lemma:var}~\eqref{var:der_strip_3}. This allows for refinement of the inequality of Lemma~\ref{lemma:h_k}~\eqref{Dhl-Dhk_2}. For all sufficiently large $n\in\N$, we get namely
	\begin{multline*}
	\norm{Dh_{r_{n+1}}(z)-Dh_{r_{n}}(z)}\\\leq2^{-n}+2^{-n}\abs{\sum_{\left\{s\colon r_{n}<k_{s}\leq r_{n+1}\right\}}(-1)^{s}}+2^{-n}\abs{\sum_{\left\{s\colon r_{n}<k_{s}\leq r_{n+1}\right\}}(-1)^{s}\frac{\langle{w,e_{k_{s}}^{\perp}}\rangle}{\langle{w,e_{k_{s}}}\rangle}}\\
	\leq 2^{-n}\left(2+\abs{\sum_{\left\{s\colon r_{n}<k_{s}\leq r_{n+1}\right\}}(-1)^{s}\frac{\langle{w,e_{k_{s}}^{\perp}}\rangle}{\langle{w,e_{k_{s}}}\rangle}}\right).
	\end{multline*}
	Combining the upper and lower bounds on $\norm{Dh_{r_{n+1}}(z)-Dh_{r_{n}}(z)}$ derived above, we deduce
	\begin{equation*}
	\abs{\sum_{\left\{s\colon r_{n}<s\leq r_{n+1}\right\}}(-1)^{s}\frac{\langle{w,e_{k_{s}}^{\perp}}\rangle}{\langle{w,e_{k_{s}}}\rangle}}\geq 2^{n}\varepsilon(n)-2
	\end{equation*}
	\hl{for all sufficiently large $n$.} Up until now we have only required the sequence $(\varepsilon(n))_{n=1}^{\infty}$ to be summable\hl{; see \eqref{eps1}}. 
	\hl{\begin{itemize}
		\item[(\mylabel{eps2}{$\varepsilon$2})] Therefore, we may now prescribe that $\varepsilon(n)=\frac{1}{n^{2}}$ for all $n\in\N$. 
	\end{itemize}}
	 The latter expression \hl{$2^{n}\varepsilon(n)-2$ in the inequality} above is then unbounded for $n\in \N$ and provides a lower bound for the supremum in \eqref{eq:sup}.
\end{proof}
We recall the definition of a martingale; see for example~\cite[p.~94]{williams_1991}.
\begin{definition}\label{def:martingale}
	Let $(\Omega,\mathcal{F},\mu)$ be a measure space and $(\mathcal{F}_{n})_{n=0}^{\infty}$ be a filtration on $(\Omega,\mathcal{F})$. A sequence $(X_{n})_{n=0}^{\infty}$ of measurable functions $X_{n}\colon \Omega\to \R$ is called a \emph{martingale} with respect to $(\mathcal{F}_{n})_{n=0}^{\infty}$ and $\mu$ if it satisfies the following conditions:
	\begin{enumerate}[(i)]
		\item\label{mart_adapted} $X_{n}\in L^{1}(\Omega,\mathcal{F}_{n},\mu)$ for each $n$. In particular, $X_{n}$ is $\mathcal{F}_{n}$-measurable for each $n$. 
		\item\label{mart_cond_exp} $\Exp[X_{n+1}|\mathcal{F}_{n}]=X_{n}$ for each $n$.			
	\end{enumerate}
	If, in \eqref{mart_cond_exp}, the equality is weakened to the inequality $\geq$ then we call $(X_{n})_{n=0}^{\infty}$ a \emph{submartingale} with respect to $(\mathcal{F}_{n})_{n=0}^{\infty}$ and $\mu$.
\end{definition}

\begin{proposition}\label{thm:mrtgl}
	Let $v\in S^{1}$, $c>0$ and $\gamma\colon I\to \R^{2}$ be a $\C^{1}$ curve with 
	\begin{equation}\label{eq:gamma_v}
	\langle{\gamma'(t),v}\rangle\geq c,\qquad \text{for all $t\in I$}.
	\end{equation}
	Let $(\Sigma_{p})_{p=0}^{\infty}$ be a filtration on $I$ and $\beta_{p}:=\Exp[\gamma'|\Sigma_{p}]$ for each $p\geq 0$. Then the sequence of functions
	\begin{equation*}
	I\to\R,\quad t\mapsto\frac{\langle{\beta_{p}(t),v^{\perp}}\rangle}{\langle{\beta_{p}(t),v}\rangle},\qquad p\geq 0,
	\end{equation*}
	is a martingale with respect to the filtration $(\Sigma_{p})_{p=0}^{\infty}$ and probability measure
	\begin{equation*}
	\mu^{v}(A):=\frac{\int_{A}\langle{\gamma',v}\rangle\,d\leb}{\int_{I}\langle{\gamma',v}\rangle\,d\leb}\, \qquad A\subseteq I.
	\end{equation*}
	Moreover
	\begin{equation*}
	\lnorm{{L^{2}(\mu^{v})}}{\frac{\langle{\beta_{p},v^{\perp}}\rangle}{\langle{\beta_{p},v}\rangle}}\leq \frac{K(\gamma)}{c} \qquad\text{for all $p\geq 0$,}
	\end{equation*}
	where $K(\gamma)$ is a constant depending only on $\gamma$.
\end{proposition}

%\begin{lemma}\label{lemma:mart1}
%	The sequence $\left(\frac{\langle{\beta_{p},v^{\perp}}\rangle}{\langle{\beta_{p},v}\rangle}\right)_{p=1}^{\infty}$ is a martingale with respect to the measure $\mu^{v}$ and filtration $(\Sigma_{p})_{p=1}^{\infty}$. Moreover,
%	\begin{equation*}
%	\left\|\frac{\langle{\beta_{p},v^{\perp}}\rangle}{\langle{\beta_{p},v}\rangle}\right\|_{L^{2}}\leq 4.
%	\end{equation*}
%\end{lemma}
\begin{proof}
	In what follows we will assume $\int_{I}\langle{\gamma',v}\rangle\,d\leb=1$, which simplifies the expression for the measure $\mu^{v}$. Accordingly all computations are correct up to \hl{multiplication} by a fixed constant $K(\gamma)$ depending only on $\gamma$. From elementary properties of the conditional expectation we get that \eqref{eq:gamma_v} implies 
	\begin{equation*}
	\langle{\beta_{p}(t),v}\rangle \geq c\qquad \text{for all $t\in I$.}
	\end{equation*}
	Hence the mappings $\frac{\langle{\beta_{p},v^{\perp}}\rangle}{\langle{\beta_{p},v}\rangle}$ are bounded, which trivially implies $\frac{\langle{\beta_{p},v^{\perp}}\rangle}{\langle{\beta_{p},v}\rangle}\in L^{1}(I,\Sigma_{p},\mu^{v})$ for every $p\geq 0$. Hence property \eqref{mart_adapted} of Definition~\ref{def:martingale} is satisfied. We turn now to property \eqref{mart_cond_exp}. Given $A\in \Sigma_{p}$ we have
	\begin{equation*}
	\int_{A}\frac{\langle{\beta_{p+1},v^{\perp}}\rangle}{\langle{\beta_{p+1},v}\rangle}\,d\mu^{v}=\int_{A}\frac{\langle{\beta_{p+1},v^{\perp}}\rangle}{\langle{\beta_{p+1},v}\rangle}\cdot \langle{\gamma',v}\rangle\,d\leb=\int_{A}\Exp\left[\frac{\langle{\beta_{p+1},v^{\perp}}\rangle}{\langle{\beta_{p+1},v}\rangle}\cdot \langle{\gamma',v}\rangle\,|\,\Sigma_{p+1}\right]\,d\leb.
	\end{equation*}
	Now we use a standard property of the conditional expectation (see \cite[22.(i), p. 54]{williams1979diffusions}) to deduce
	\begin{equation*}
	\Exp\left[\frac{\langle{\beta_{p+1},v^{\perp}}\rangle}{\langle{\beta_{p+1},v}\rangle}\cdot \langle{\gamma',v}\rangle\,|\,\Sigma_{p+1}\right]=\frac{\langle{\beta_{p+1},v^{\perp}}\rangle}{\langle{\beta_{p+1},v}\rangle}\cdot\Exp[\langle{\gamma',v}\rangle|\Sigma_{p+1}]=\langle{\beta_{p+1},v^{\perp}}\rangle,
	\end{equation*}
	and similarly
	\begin{equation*}
	\Exp\left[\frac{\langle{\beta_{p},v^{\perp}}\rangle}{\langle{\beta_{p},v}\rangle}\cdot \langle{\gamma',v}\rangle\,|\,\Sigma_{p}\right]=\langle{\beta_{p},v^{\perp}}\rangle.
	\end{equation*}
	Hence
	\begin{align*}
	\int_{A}\frac{\langle{\beta_{p+1},v^{\perp}}\rangle}{\langle{\beta_{p+1},v}\rangle}\,d\mu^{v}&=\int_{A}\langle{\beta_{p+1},v^{\perp}}\rangle\, d\leb=\int_{A}\langle{\gamma',v^{\perp}}\rangle\,d\leb=\int_{A}\langle{\beta_{p},v^{\perp}}\rangle\,d\leb\\
	&=\int_{A}\frac{\langle{\beta_{p},v^{\perp}}\rangle}{\langle{\beta_{p},v}\rangle}\cdot \langle{\gamma',v}\rangle\, d\leb=\int_{A}\frac{\langle{\beta_{p},v^{\perp}}\rangle}{\langle{\beta_{p},v}\rangle}\,d\mu^{v}.
	\end{align*}
	The bound on the $L^{2}(\mu^{v})$ norm follows trivially from a bound on the $L^{\infty}$ norm:
	\begin{equation*}
	\lnorm{\infty}{\frac{\langle{\beta_{p},v^{\perp}}\rangle}{\langle{\beta_{p},v}\rangle}}\leq\frac{1}{c}\quad\Rightarrow\quad\lnorm{{L^{2}(\mu^{v})}}{\frac{\langle{\beta_{p},v^{\perp}}\rangle}{\langle{\beta_{p},v}\rangle}}\leq \frac{K(\gamma)}{c}.
	\end{equation*}
\end{proof}
The proof of the next lemma can be given as an exercise; we include it in \hl{Appendix~\ref{app_mart}}.
\begin{restatable}{lemma}{lemmaaltsummart}\label{lemma:alt_sum_mart}
	Let $(\Omega,\mathcal{F},\mu)$ be a measure space, $(\mathcal{F}_{n})_{n=0}^{\infty}$ be a filtration on $\Omega$ and $(X_{n})_{n=0}^{\infty}$ be a martingale with respect to the filtration $(\mathcal{F}_{n})_{n=0}^{\infty}$ and measure $\mu$. Then the sequence of alternating sums
	\begin{equation*}
	\sum_{n=0}^{2N-1}(-1)^{n}X_{n},\qquad N\in\N,
	\end{equation*}
	is a martingale with respect to the filtration $(\mathcal{F}_{2N-1})_{N=1}^{\infty}$ and measure $\mu$ with 
	\begin{equation*}
	\lnorm{{L^{2}(\mu)}}{\sum_{n=0}^{2N-1}(-1)^{n}X_{n}}\leq 2 \sup_{n\geq 0}\lnorm{{L^{2}(\mu)}}{X_{n}}.
	\end{equation*}
\end{restatable}
Together Proposition~\ref{thm:mrtgl} and Lemma~\hl{\ref{lemma:alt_sum_mart}} admit the following corollary:
\begin{corollary}\label{cor:kolm}
	With the hypothesis of Proposition~\ref{thm:mrtgl} we have for every $\lambda>0$
	\begin{equation*}
	\leb\left(\left\{t\in I\colon \sup_{p\in\N}\abs{\sum_{q=0}^{2p-1}(-1)^{q}\frac{\langle{\beta_{q}(t),v^{\perp}}\rangle}{\langle{\beta_{q}(t),v}\rangle}}>\lambda\right\}\right)\leq \frac{16 K(\gamma)}{\lambda^{2}c^{3}}
	\end{equation*}	
\end{corollary}
\begin{proof}
	By combining Proposition~\ref{thm:mrtgl} and Lemma~\ref{lemma:alt_sum_mart} we deduce that the sequence of functions
	\begin{equation*}
	\alpha_{p}^{v}\colon I\to\R,\qquad t\mapsto\sum_{q=0}^{2p-1}(-1)^{q}\frac{\langle{\beta_{q}(t),v^{\perp}}\rangle}{\langle{\beta_{q}(t),v}\rangle},\qquad p\geq 1,
	\end{equation*}
	is a martingale with respect to the filtration $(\Sigma_{2p-1})_{p=1}^{\infty}$ and measure $\mu^{v}$ with
	\begin{equation*}
	\left\|\alpha_{p}^{v}\right\|_{L^{2}(\mu^{v})}\leq \frac{2}{c}\qquad \text{for all $p\geq 1$.}
	\end{equation*}
	Now, making use of Doob's $L^{2}$ inequality \cite[p.~60]{williams1979diffusions}, we derive
	\begin{multline*}
	\lambda^{2}\mu^{v}\left(\left\{t\in I\colon \sup_{p\in\N}\abs{\sum_{q=0}^{2p-1}(-1)^{q}\frac{\langle{\beta_{q}(t),v^{\perp}}\rangle}{\langle{\beta_{q}(t),v}\rangle}}>\lambda\right\}\right)\\\leq \lnorm{{L^{2}(\mu^{v})}}{\sup_{p\in\N}\abs{\alpha_{p}^{v}}}^{2}\leq 2^{2}\sup_{q\in \N}\lnorm{{L^{2}(\mu^{v})}}{\alpha_{p}^{v}}^{2}\leq \frac{16}{c^{2}},
	\end{multline*}
	after which a simple rearrangement and application of the inequality $\leb\leq \frac{K(\gamma)}{c}\mu^{v}$ verifies the corollary.

\end{proof}

\begin{lemma}\label{lemma:preimage_H_zero}
	Let $\gamma\colon I\to\R^{2}$ be a $\C^{1}$ curve with
	\begin{equation*}
	\gamma'(t)\in C(w,2\eta)\qquad\text{for all $t\in I$.}
	\end{equation*}
	Then $\leb(\gamma^{-1}(H\setminus G))=0$.
\end{lemma}
\begin{proof}
	At this point we \hl{prescribe} that $\eta$ is sufficiently small so that the conditions of Hypothesis~\ref{hyp} are satisfied for $\delta=2\eta$ and the conditions of Proposition~\ref{thm:mrtgl} are satisfied for $c=1-2\eta$ and $v=w$. Let the filtration $(\Sigma_{p})_{p=0}^{\infty}$, the conditional expectations $\beta_{p}:=\Exp[\gamma'|\Sigma_{p}]$ and the set $D\subseteq I$ be defined according to Hypothesis~\ref{hyp}. In view of Lemma~\ref{prop:Dp} it suffices to show that the set
	\begin{equation*}
	Z:=\gamma^{-1}(H\setminus G)\setminus D
	\end{equation*}
	has Lebesgue measure zero. Let $t\in Z$. 
%	Since $t\notin D$ there exists $N\in\N$ such that 
%	\begin{equation*}
%	\gamma(t)\notin \bigcup_{k=N}^{\infty}B(S_{k},\delta_{k})\qquad \text{and}\qquad t\notin \bigcup_{p=N}^{\infty}D_{p}.
%	\end{equation*}
	Applying Lemma~\ref{lemma:der_grow} with $z=\gamma(t)\in H\setminus G$ we get, writing $k_{s}$ for $k_{s}(\gamma(t))$,
	\begin{equation*}
	\sup_{q\in\N}\left|\sum_{s=0}^{ 2q-1}(-1)^{s}\frac{\langle{w,e_{k_{s}}^{\perp}}\rangle}{\langle{w,e_{k_{s}}}\rangle}\right|=\infty,
	\end{equation*}
	which together with Lemma~\ref{lemma:outsideDp} and $t\notin D$ implies 
	\begin{equation*}
	\sup_{q\in\N}\left|\sum_{s=0}^{2q-1}(-1)^{s}\frac{\langle{\beta_{s}(t),w^{\perp}}
		\rangle}{\langle{\beta_{s}(t),w}\rangle}\right|=\infty.
	\end{equation*}
	To summarise, we have shown that
	\begin{equation*}
	Z\subseteq \left\{t\in I\colon \sup_{q\in\N}\left|\sum_{s=0}^{2q-1}(-1)^{s}\frac{\langle{\beta_{s}(t),w^{\perp}}
		\rangle}{\langle{\beta_{s}(t),w}\rangle}\right|=\infty\right\},
	\end{equation*}
	and the latter set has measure zero by Corollary~\ref{cor:kolm}.
\end{proof}
The following statement is the final piece in the proof of Theorem~\ref{thm:main_result}.
\begin{lemma}\label{lemma:final_piece}
\hl{The set $G\cup H$ is purely unrectifiable.}
\end{lemma}
\begin{proof}

	Both $G$ and $H$ are $G_{\delta}$ sets, hence $G\cup H$ is Borel. Let $\gamma\colon [0,1]\to\R^{2}$ be a $\mathcal{C}^{1}$ curve. To complete the proof we verify that the set $\gamma^{-1}(G\cup H)$ has Lebesgue measure zero. We may cover $[0,1]$ by countably many intervals $I$ so that each restriction $\gamma\colon I\to\R^{2}$ satisfies (possibly with orientation reversed) either
	\begin{equation*}
	\gamma'(t)\notin \widehat{C}(w,\eta)\quad\text{ for all $t\in I$},\qquad\text{or, }\gamma'(t)\in C(w,2\eta)\quad\text{ for all $t\in I.$}
	\end{equation*}
	It now suffices to argue that each such restriction of $\gamma$ intersects $G\cup H$ in a set of measure zero. The curves for which the first condition holds intersect $E\supset G\cup H$ in a set of measure zero, by Lemma~\ref{lemma:easy_curves}. The curves of the second type intersect $H\setminus G$ in a set of measure zero, by Lemma~\ref{lemma:preimage_H_zero} and $G$ in a set of measure zero by Lemma~\ref{lemma:G_zero}.
\end{proof}

\hl{\begin{proof}[Proof of Theorem~\ref{thm:main_result}]
	 Let $E$ be the universal differentiabilty set given by \eqref{eq:set_E} and the construction that follows. We take the parameter $\eta\in (0,1)$ of \eqref{E1} to be sufficiently small so that $3\sqrt{\eta}<\alpha$. Let $h\colon\R^{2}\to\R$ be the Lipschitz function corresponding to $E$ constructed in Section~\ref{subsec:function}; see \eqref{eq:function_h}. Then Lemma~\ref{lemma:final_piece} and Lemma~\ref{lemma:cover_diff} together establish that $h$ verifies the conclusion of Theorem~\ref{thm:main_result}.
	\end{proof}}

\subsection{Proof of Theorem~\ref{thm:pu_tds}}
Referring to the above construction, we present a proof of Theorem~\ref{thm:pu_tds}. 
\begin{proof}[Proof of Theorem~\ref{thm:pu_tds}]
	We begin with the universal differentiability set $E$ \hl{of \eqref{eq:set_E}} and perform the following trimmings. First we appeal to Lemma~\ref{lemma:compact_uds} to replace $E$ with a compact universal differentiability set $\hl{Y}\subseteq E\cap[0,1]^{2}$. Next, we remove the set $G$ and argue that $G$ is sufficiently negligible so that we again retain a universal differentiability set. This is justified by \cite[Lemma~2.1]{dymond_maleva2016} and the fact, of Lemma~\ref{lemma:G_zero}, that $G$ projects in any direction to a set of $1$-dimensional Lebesgue measure zero. Thus, in the end, we are left with a universal differentiability set
	\begin{equation*}
	\hl{\widetilde{Y}}:=\hl{Y}\setminus G\subseteq [0,1]^{2}.
	\end{equation*}
	In light of Lemmas~\ref{lemma:cover_diff} and \ref{lemma:h_diff} we have that $h$ is non-differentiable at all points of $\hl{\widetilde{Y}}\setminus H\subseteq E\setminus(G\cup H)$ and differentiable at all points of $\hl{\widetilde{Y}}\cap H\subseteq H\setminus G$. We verify that the set $P:=\hl{\widetilde{Y}}\cap H$ has the properties asserted in Theorem~\ref{thm:pu_tds}.
	
	\hl{First, note that $P$ is purely unrectifiable, due to Lemma~\ref{lemma:final_piece}. It remains to show that typical functions $g\in\lip_{1}([0,1]^{2})$ have large sets of differentiability points in $P$ in the senses of Theorem~\ref{thm:pu_tds}~\eqref{non_sig_fin} and \eqref{all_proj}.} Observe that $\hl{\widetilde{Y}}\setminus H\subseteq \hl{Y}\setminus (G\cup H)=\hl{Y}\cap\bigcup_{m=1}^{\infty}F_{m}$. Since $\hl{Y}$ is compact and each $F_{m}$ is $F_{\sigma}$, the sets $\hl{Y}\cap F_{m}$ are $F_{\sigma}$. Moreover, for each $m$\hl{,} Lemma~\ref{lemma:cover_diff} ensures the conditions of Lemma~\ref{lemma:Fsig_unif_ndiff} are satisfied for $K=\hl{Y}\cap F_{m}$, $\sigma=\frac{2^{-m}\sqrt{\eta}}{4}$ and $h$. Intersecting the residual subsets of $\lip_{1}([0,1]^{2})$ obtained by applying Lemma~\ref{lemma:Fsig_unif_ndiff} to each $\hl{Y}\cap F_{m}$, we obtain a residual set in which all functions $g$ have the property that $g+h$ is nowhere differentiable in $\hl{Y}\cap \bigcup_{m=1}^{\infty}F_{m}\supseteq \hl{\widetilde{Y}}\setminus H$. Since $\hl{\widetilde{Y}}$ is a universal differentiability set, it follows that 
	\begin{equation*}
	\emptyset\neq \Diff(g+h)\cap \hl{\widetilde{Y}}\subseteq \hl{\widetilde{Y}}\cap H=P
	\end{equation*}
	for typical $g\in\lip_{1}([0,1]^{2})$. But $h$ is differentiable at all points of $P=\hl{\widetilde{Y}}\cap H$, so we conclude that 
	\begin{equation*}
	\Diff(g)\cap P\supseteq \Diff(g+h)\cap \hl{\widetilde{Y}}
	\end{equation*}
	for typical $g\in \lip_{1}([0,1]^{2})$. The latter sets \hl{are Borel (see~\cite[Corollary~3.5.5]{LPT2012frechet}), purely unrectifiable and} have all one-dimensional projections of positive measure by \cite[Lemma~2.1]{dymond_maleva2016}. \hl{Therefore, by} the Besicovitch-Federer Projection Theorem~\hl{\cite[Theorem~18.1]{mattila_1995}}, they must also be of non-$\sigma$-finite one-dimensional Hausdorff measure.
\end{proof}

\paragraph{Acknowledgements.} The author would like to thank Olga Maleva and David Preiss for helpful discussions. The research presented in this paper was supported in part by short research visits at the University of Birmingham and the author wishes to thank the School of Mathematics for their hospitality. The author acknowledges the support of Austrian Science Fund (FWF): P 30902-N35.

\appendix
\section{Appendix}
\subsection{Geometry of curves}\label{app_geom_curves}
For $W\subseteq\R^{2}$ and $v\in S^{1}$ we define a quantity
\begin{equation*}
\diam_{v}(W):=\sup\left\{\langle{y-x,v}\rangle\colon x,y\in W\right\}.
\end{equation*}
\begin{lemma}\label{lemma:crossing}
	Let $W\subseteq \R^{2}$, $v\in S^{1}$, $\delta\in(0,1)$ and $\gamma\colon I\to \R^{2}$ be a $\C^{1}$~curve satisfying
	\begin{equation*}
	\gamma'(t)\in \widehat{C}(v,\delta)\quad \text{for all $t\in I$.}
	\end{equation*}
	Then 
	\begin{equation*}
	\leb(\gamma^{-1}(W))\leq \frac{\diam_{v}(W)}{1-\delta}.
	\end{equation*}	
\end{lemma}
\begin{proof}
	Since $\gamma'$ is continuous, we either have $\gamma'(t)\in C(v,\delta)$ for all $t\in I$ or $\gamma'(t)\in C(-v,\delta)$ for all $t\in I$. We assume the former without loss of generality.
	Then the function $t\mapsto \langle{\gamma(t),v}\rangle$ is strictly increasing, implying that the set $\gamma^{-1}(W)$ is contained in the interval $[a,b]$, where $a,b\in \gamma^{-1}(\overline{W})$ are defined by the conditions
	\begin{align*}
	\langle{\gamma(a),v}\rangle=\min\left\{\langle{\gamma(t),v}\rangle\colon t\in\gamma^{-1}(\overline{W})\right\},\quad
	\langle{\gamma(b),v}\rangle=\max\left\{\langle{\gamma(t),v}\rangle\colon t\in\gamma^{-1}(\overline{W})\right\}.
	\end{align*}
	Now we have
	\begin{equation*}
	\leb(\gamma^{-1}(W))\leq b-a\leq \frac{1}{1-\delta}\int_{a}^{b}\langle{\gamma'(t),v}\rangle\,dt=\frac{\langle{\gamma(b)-\gamma(a),v}\rangle}{1-\delta}\leq \frac{\diam_{v}(W)}{1-\delta}.
	\end{equation*}
\end{proof}

\begin{lemma}\label{lemma:convex_curve}
	Let $P\subseteq \R^{2}$ be an open and convex set, $e\in S^{1}$ be a direction and $\gamma\colon I\to\R^{2}$ be a $\C^{1}$ curve with $\langle{\gamma'(t),e}\rangle\geq 0$ for all $t\in I$. Then
	\begin{equation*}
	\left|\int_{\gamma^{-1}(P)}\langle{\gamma'(t),e^{\perp}}\rangle\,dt\right|\leq 6\diam_{e^{\perp}}(P).
	\end{equation*}
\end{lemma}
\begin{proof}
	Let $a:=\inf\left\{\langle{z,e}\rangle\colon z\in P\right\}$ and $b:=\sup\left\{\langle{z,e}\rangle \colon z\in P\right\}$. As a convex and open set, $P$ admits functions $\psi^{-},\psi^{+}\colon (a,b)\to \R$ with $\psi^{-}$ convex and $\psi^{+}$ concave, $\psi^{-}<\psi^{+}$ and so that $\partial P\cap\left\{z\in\R^{d}\colon a<\langle{z,e}\rangle<b\right\}$ is the union of the graphs of $\psi^{-}$ and $\psi^{+}$ in the co-ordinate system $(e,e^{\perp})$. For points $z\in\R^{2}$ with $a<\langle{z,e}\rangle<b$ and $\psi\in\set{\psi^{+},\psi^{-}}$ we will let, for example, $z\geq \psi$ signify that, with respect to the coordinate system $(e,e^{\perp})$, the point $z$ lies on or above the graph of $\psi\colon (a,b)\to\R$. With this notation we have
	\begin{equation*}
	P\cap\left\{z\in\R^{d}\colon a<\langle{z,e}\rangle<b\right\}=\set{z\colon \psi^{-}<z<\psi^{+}}.
	\end{equation*}
	The condition $\langle{\gamma'(t),e}\rangle\geq 0$ guarantees that the segment $\gamma(I)\cap \set{z\in \R^{2}\colon a\leq \langle{z,e}\rangle\leq b}$ is connected. This leads to the simple observation that whenever $s<t$ with $\gamma(s)\geq \psi$ and $\gamma(t)<\psi$ there must be a point $r\in[s,t]$ with $\gamma(r)\in\Graph(\psi)$. We make frequent use of this observation in the argument that follows.

	The open set $\gamma^{-1}(P)$ can be written as a countable union of intervals $(a_{2i-1},a_{2i})\subseteq I$, $i=1,2,\ldots$ with $\gamma(a_{j})\in\partial P$ for all $j$. We choose $N\in\N$ sufficiently large so that
	\begin{equation*}
	\sum_{i\geq N+1}(a_{2i}-a_{2i-1})\leq \diam_{e^{\perp}}(P).
	\end{equation*} 
	By relabelling if necessary, we may assume that $a_{1}<a_{2}\leq a_{3}<a_{4}\leq \ldots \leq a_{2N-1}<a_{2N}$. In what follows we say that the point $a_{j}$ is of type $+$, respectively of type $-$, if $\gamma(a_{j})\in \Graph{\psi^{+}}$, respectively if $\gamma(a_{j})\in \Graph{\psi^{-}}$. We argue that the finite sequence $((a_{2j-1},a_{2j}))_{j=1}^{N}$ may be extended to a finite sequence of \hl{connected} components of $\gamma^{-1}(P)$ ordered with respect to $\leq$ in which the points $a_{2j+1}$ and $a_{2j}$ have the same type for every $j$. Let $j\in\left\{1,\ldots,N-1\right\}$ be an index for which $a_{2j}$ and $a_{2j+1}$ have different types. Without loss of generality, we may assume that $\gamma(a_{2j})\in\Graph{\psi^{+}}$ and $\gamma(a_{2j+1})\in\Graph{\psi^{-}}$. Then the interval \hl{$[a_{2j},a_{2j+1}]$} must contain a \hl{connected} component $(a_{2k-1},a_{2k})$ of $\gamma^{-1}(P)$ with $k\geq N+1$, $a_{2k-1}\in \Graph{\psi^{+}}$ and $a_{2k}\in \Graph{\psi^{-}}$. For example, the points $a_{2k-1}$ and $a_{2k}$ may be defined by
	\begin{align*}
	a_{2k}&:=\inf\left\{t>a_{2j}\colon \gamma(t)\leq \psi^{-}\right\}\leq a_{2j+1},\qquad
	a_{2k-1}:=\sup\left\{t<a_{2k}\colon \gamma(t)\geq \psi^{+}\right\}\geq a_{2j}.
	\end{align*}
	For each index $j\in\left\{1,\ldots,N-1\right\}$ for which $a_{2j}$ and $a_{2j+1}$ have different types, we let $k_{j}\geq N+1$ be the index defined by the above discussion. By inserting the intervals $(a_{2k_{j}-1},a_{2k_{j}})$ in between the relevant terms of the original sequence $(a_{2i-1},a_{2i})_{i=1}^{N}$ and relabelling the \hl{connected} components of $\gamma^{-1}(P)$ we obtain an extended finite sequence $((a_{2i-1},a_{2i}))_{i=1}^{M}$ of \hl{connected} components of $\gamma^{-1}(P)$ ordered with respect to $\leq$ with the desired property. 
	
	Thus, for each $i\in\left\{1,\ldots,M-1\right\}$ either $\gamma(a_{2i})$ and $\gamma(a_{2i+1})$ both lie on the graph of $\psi^{+}$ or they both lie on the graph of $\psi^{-}$. The sum of the quantities
	\begin{equation*}
	\langle{\gamma(a_{2i+1})-\gamma(a_{2i}),e^{\perp}}\rangle=\psi^{-}(\langle{\gamma(a_{2i+1}),e}\rangle)-\psi^{-}(\langle{\gamma(a_{2i}),e}\rangle)
	\end{equation*}
	over $1\leq i\leq M$ for which the first case occurs is bounded above by $2\diam_{e^{\perp}}(P)$, because $\psi^{-}$ is convex and oscillates at most $\diam\left\{\langle{z,e^{\perp}}\rangle\colon z\in P\right\}=\diam_{e^{\perp}}(P)$. The same estimate holds for the corresponding sum over the second case indices. This gives us
	\begin{multline*}
	\left|\int_{\gamma^{-1}(P)}\langle{\gamma'(t),e^{\perp}}\rangle\,dt\right|\leq\left|\sum_{i=1}^{M}\langle{\gamma(a_{2i})-\gamma(a_{2i-1}),e^{\perp}}\rangle\right|+\sum_{i\geq M+1}(a_{2i}-a_{2i-1})\\\leq \left|\langle{\gamma(a_{2M})-\gamma(a_{1}),e^{\perp}}\rangle\right|+\left|\sum_{i=1}^{M-1}\langle{\gamma(a_{2i+1})-\gamma(a_{2i}),e^{\perp}}\rangle\right|+\diam_{e^{\perp}}(P)\\
	\leq \diam_{e^{\perp}}(P)+4\diam_{e^{\perp}}(P)+\diam_{e^{\perp}}(P)=6\diam_{e^{\perp}}(P).
	\end{multline*}
\end{proof}

\subsection{Martingale Theory}\label{app_mart}

\lemmaaltsummart*
\begin{proof}
	For an arbitrary set $A\in \mathcal{F}_{2N-1}$ we have
	\begin{align*}
	\int_{A}\sum_{n=0}^{2N+1}(-1)^{n}X_{n}\,d\mu&=\int_{A}\sum_{n=0}^{2N-1}(-1)^{n}X_{n}\,d\mu+\int_{A}X_{2N}\,d\mu-\int_{A}X_{2N+1}\,d\mu\\
	%&=\int_{A}\sum_{n=0}^{2N-1}(-1)^{n}X_{n}\,d\mu+\int_{A}X_{2N-1}\,d\mu-\int_{A}X_{2N-1}\,d\mu\\
	&=\int_{A}\sum_{n=0}^{2N-1}(-1)^{n}X_{n}\,d\mu
	\end{align*}
	This proves 
	\begin{equation*}
	\mathbb{E}\left[\sum_{n=0}^{2N+1}(-1)^{n}X_{n}|\mathcal{F}_{2N-1}\right]=\sum_{n=0}^{2N-1}(-1)^{n}X_{n},\qquad N\in\N.
	\end{equation*}
	This establishes the martingale part. To get the bound on the $L^{2}$ norm we note that for $n\geq m$ we have
	\begin{equation*}
	\langle{X_{m},X_{n}}\rangle=\int_{\Omega}X_{m}X_{n}\,d\mu=\int_{\Omega}\mathbb{E}[X_{m}X_{n}|\mathcal{F}_{m}]\,d\mu=\int_{\Omega}X_{m}\mathbb{E}[X_{n}|\mathcal{F}_{m}]\,d\mu=\int_{\Omega}X_{m}^{2}\,d\mu,
	\end{equation*}
	where $\langle{-,-}\rangle$ denotes the standard inner product on $L^{2}(\Omega,\mathcal{F},\mu)$. The third equality above makes use of a standard property of the conditional expectation~\cite[22.(i), p~54]{williams1979diffusions}. We may now compute the $L^{2}$ norm of the alternating sum as follows
	\begin{align*}
	\left\|\sum_{n=0}^{2N-1}(-1)^{n}X_{n}\right\|_{L^{2}}^{2}&=\sum_{0\leq m,n\leq 2N-1}(-1)^{m+n}\langle{X_{m},X_{n}}\rangle\\
	&=2\sum_{0\leq m\leq n\leq 2N-1}(-1)^{m+n}\langle{X_{m},X_{n}}\rangle-\sum_{0\leq n\leq 2N-1}(-1)^{2n}\langle{X_{n},X_{n}}\rangle\\
	&=2\sum_{m=0}^{2N-1}(-1)^{m}\langle{X_{m},X_{m}}\rangle\sum_{n=m}^{2N-1}(-1)^{n}-\sum_{m=0}^{2N-1}\langle{X_{m},X_{m}}\rangle\\
	&=2\sum_{m\text{ odd}}(-1)^{2m}\langle{X_{m},X_{m}}\rangle -\sum_{m=0}^{2N-1}\langle{X_{m},X_{m}}\rangle\\
	&=\sum_{m=0}^{2N-1}(-1)^{m+1}\langle{X_{m},X_{m}}\rangle\\
	&=\sum_{m=0}^{2N-1}(-1)^{m+1}\int_{\Omega}X_{m}^{2}\,d\mu.
	\end{align*}
	The sequence $(X_{m}^{2})_{m=1}^{\infty}$ is a submartingale; hence the inequality
	\begin{equation*}
	\int_{\Omega}X_{2n-1}^{2}\,d\mu\leq \int_{\Omega}X_{2n}^{2}\,d\mu,\qquad n\geq 1
	\end{equation*}
	holds. Applying this inequality to the final expression above we deduce
	\begin{equation*}
	\left\|\sum_{n=0}^{2N-1}(-1)^{n}X_{n}\right\|_{L^{2}}^{2}\leq \int_{\Omega}X_{2N-1}^{2}\,d\mu-\int_{\Omega}X_{0}^{2}\,d\mu\leq 2\sup_{n\in\N}\left\|X_{n}\right\|_{L^{2}}^{2}.
	\end{equation*}
	\end{proof}
\bibliographystyle{plain}
\bibliography{biblio}

\begin{thebibliography}{10}

\bibitem{acp2010differentiability}
G.~Alberti, M.~Cs{\"o}rnyei, and D.~Preiss.
\newblock {Differentiability of Lipschitz functions, structure of null sets,
  and other problems}.
\newblock In {\em Proceedings of the International Congress of Mathematicians
  2010 (ICM 2010) (In 4 Volumes) Vol. I: Plenary Lectures and Ceremonies Vols.
  II--IV: Invited Lectures}, pages 1379--1394. World Scientific, 2010.

\bibitem{CJ}
M.~Csörnyei and P.~Jones.
\newblock { Product Formulas for Measures and Applications to Analysis and
  Geometry.}
\newblock {\em URL: ww.math.sunysb.edu/Videos/dfest/PDFs/38-Jones.pdf.}

\bibitem{CPT_2005}
M.~Csörnyei, D.~Preiss, and J.~Tiser.
\newblock Lipschitz functions with unexpectedly large sets of
  nondifferentiability points.
\newblock 2005, 01 2005.

\bibitem{Dore_Maleva1}
M.~Dor{\'e} and O.~Maleva.
\newblock {A compact null set containing a differentiability point of every
  Lipschitz function}.
\newblock {\em Mathematische Annalen}, 351(3):633--663, Nov 2011.

\bibitem{Dore_Maleva2}
M.~Dor{\'e} and O.~Maleva.
\newblock {A compact universal differentiability set with Hausdorff dimension
  one}.
\newblock {\em Israel Journal of Mathematics}, 191(2):889--900, Oct 2012.

\bibitem{Dore_Maleva3}
M.~Doré and O.~Maleva.
\newblock {A universal differentiability set in Banach spaces with separable
  dual}.
\newblock {\em Journal of Functional Analysis}, 261(6):1674 -- 1710, 2011.

\bibitem{dymond2017structure}
M.~Dymond.
\newblock On the structure of universal differentiability sets.
\newblock {\em Comment. Math. Univ. Carolin}, 58(3):315--326, 2017.

\bibitem{dymond_maleva2016}
M.~Dymond and O.~Maleva.
\newblock {Differentiability inside sets with Minkowski dimension one}.
\newblock {\em Michigan Math. J.}, 65(3):613--636, 08 2016.

\bibitem{dymond_maleva2019dichotomy}
Michael Dymond and Olga Maleva.
\newblock A dichotomy of sets via typical differentiability.
\newblock {\em arXiv preprint arXiv:1909.03487}, 2019.

\bibitem{kechris2012classical}
A.~Kechris.
\newblock {\em Classical descriptive set theory}, volume 156.
\newblock Springer Science \& Business Media, 2012.

\bibitem{ledonne_pinamonti_speight2017universal}
E.~Le~Donne, A.~Pinamonti, and G.~Speight.
\newblock {Universal differentiability sets and maximal directional derivatives
  in Carnot groups}.
\newblock {\em Journal de Math{\'e}matiques Pures et Appliqu{\'e}es}, 2017.

\bibitem{LPT2012frechet}
J.~Lindenstrauss, D.~Preiss, and J.~Ti{\v{s}}er.
\newblock {\em Fr{\'e}chet differentiability of Lipschitz functions and porous
  sets in Banach spaces}.
\newblock Princeton University Press, 2012.

\bibitem{maleva_preiss2018}
O.~Maleva and D.~Preiss.
\newblock {Cone unrectifiable sets and non-differentiability of Lipschitz
  functions}.
\newblock {\em Israel Journal of Mathematics}, 8 2018.

\bibitem{mattila_1995}
P.~Mattila.
\newblock {\em {Geometry of Sets and Measures in Euclidean Spaces: Fractals and
  Rectifiability}}.
\newblock Cambridge Studies in Advanced Mathematics. Cambridge University
  Press, 1995.

\bibitem{merlo}
A.~Merlo.
\newblock {Full non-differentiability of typical Lipschitz functions}.
\newblock {\em in preparation}.

\bibitem{pinamonti_speight2017uds}
A.~Pinamonti and G.~Speight.
\newblock {A measure zero universal differentiability set in the Heisenberg
  group}.
\newblock {\em Mathematische Annalen}, 368(1-2):233--278, 2017.

\bibitem{PREISS1990312}
D~Preiss.
\newblock {Differentiability of Lipschitz functions on Banach spaces}.
\newblock {\em Journal of Functional Analysis}, 91(2):312 -- 345, 1990.

\bibitem{preiss_speight2013}
D.~Preiss and G.~Speight.
\newblock {Differentiability of Lipschitz Functions in Lebesgue Null Sets}.
\newblock {\em Inventiones mathematicae}, 197, 2013.

\bibitem{preiss_tiser94}
D.~Preiss and J.~Tišer.
\newblock {Points of non-differentiability of typical Lipschitz functions}.
\newblock {\em Real Analysis Exchange}, 20(1):219--226, 1994.

\bibitem{williams1979diffusions}
D.~Williams.
\newblock {\em Diffusions, Markov processes, and martingales. Vol. 1,
  Foundations}.
\newblock Wiley, 1979.

\bibitem{williams_1991}
D.~Williams.
\newblock {\em {Probability with Martingales}}.
\newblock Cambridge University Press, 1991.

\end{thebibliography}
\end{document}